\documentclass[a4paper, 11pt]{extarticle} 
\usepackage[a4paper,left=2cm,right=2cm,top=2cm,bottom=2cm]{geometry}
\usepackage{amsmath,amssymb,amsthm}
\usepackage{xparse}
\usepackage[dvipsnames]{xcolor}
\usepackage{enumitem}
\usepackage{wasysym}
\usepackage{dsfont}
\usepackage{lmodern}
\usepackage{soul}
\usepackage{tikz}
\usetikzlibrary{calc,positioning}

\usepackage{float}
\usepackage[pdftex,bookmarksopen=true,bookmarks=true,unicode,setpagesize]{hyperref}
\hypersetup{colorlinks=true,linkcolor=black,citecolor=black,urlcolor=blue}
\usepackage{authblk}

\usepackage{graphicx}

\def\txtd{{\textnormal{d}}}

\def\txtD{{\textnormal{D}}}

\numberwithin{equation}{section}
\numberwithin{figure}{section}
\newcounter{dummy} \numberwithin{dummy}{section}
\newtheorem{thm}[dummy]{Theorem}
\newtheorem{lem}[dummy]{Lemma}
\newtheorem{prop}[dummy]{Proposition}
\newtheorem{cor}[dummy]{Corollary}

\newtheorem{exmp}[dummy]{Example}
\newtheorem{rem}[dummy]{Remark}

\newcommand{\cO}{{\mathcal O}}  

\newcounter{assum}
\newenvironment{assum}{\refstepcounter{assum}\equation}{\tag{A\theassum}\endequation}
\newcounter{hypoth}
\newtheorem{hypothesis}[hypoth]{Hypothesis} 

\newcounter{cque}

\binoppenalty=\maxdimen
\relpenalty=\maxdimen

\date{\today}

\newcommand{\R}{\mathbb{R}}
\renewcommand{\C}{\mathbb{C}}
\newcommand{\X}{{\mathbb{R}^d}}
\newcommand{\N}{\mathbb{N}}
\newcommand{\Z}{\mathbb{Z}}

\newcommand{\F}{\mathcal{F}}

\newcommand{\eps}{\varepsilon}
\newcommand{\la}{\lambda}

\newcommand{\ka}{\varkappa}

\newcommand{\kap}{\varkappa^+}
\newcommand{\kape}{\varkappa^+_\varepsilon}

\newcommand{\kam}{\varkappa^-}
\newcommand{\kame}{\varkappa^-_\varepsilon}

\newcommand{\Aek}{\mathcal{A}_{\eps,k}}
\newcommand{\Ac}{\mathcal{A}_c}

\newcommand{\at}{\widetilde{\alpha}}
\newcommand{\1}{\mathds{1}}

\newcommand{\Buc}{BUC(\R)}
\newcommand{\e}{\textnormal{e}}
\renewcommand{\i}{\textnormal{i}}

\renewcommand{\Re}{\mathfrak{Re}}
\renewcommand{\Im}{\mathfrak{Im}}
\let\copyint\int

\RenewDocumentCommand \int {o o}
{ \IfNoValueTF {#2} { \IfNoValueTF {#1} { \copyint } { \copyint\limits_{#1} } }	{ \copyint\limits_{#1}^{#2} } }

\NewDocumentCommand \diff {m m}
{ \frac{\partial #1}{\partial #2} }

\author[1]{Christian Kuehn \thanks{Email: \texttt{ckuehn@ma.tum.de}}}
\author[2]{Pasha Tkachov \thanks{Email: \texttt{pasha.tkachov@gssi.it}}}

\affil[1]{Technical University of Munich}
\affil[2]{Gran Sasso Science Institute, L'Aquila}

\title{Pattern formation in the doubly-nonlocal Fisher-KPP equation.}
\begin{document}
\maketitle

\begin{abstract}
We study the existence, bifurcations, and stability of stationary solutions for the doubly-nonlocal Fisher-KPP equation. We prove using Lyapunov-Schmidt reduction that under suitable conditions on the parameters, a bifurcation from the non-trivial homogeneous state can occur. The kernel of the linearized operator at the bifurcation is two-dimensional and periodic stationary patterns are generated. Then we prove that these patterns are, again under suitable conditions, locally asymptotically stable. We also compare our results to previous work on the nonlocal Fisher-KPP equation containing a local diffusion term and a nonlocal reaction term. If the diffusion is approximated by a nonlocal kernel, we show that our results are consistent and reduce to the local ones in the local singular diffusion limit. Furthermore, we prove that there are parameter regimes, where no bifurcations can occur for the doubly-nonlocal Fisher-KPP equation. The results demonstrate that intricate different parameter regimes are possible. In summary, our results provide a very detailed classification of the multi-parameter dependence of the stationary solutions for the doubly-nonlocal Fisher-KPP equation. 
\end{abstract}


\section{Introduction}

The aim of this paper is to study existence and stability of stationary solutions to the doubly-nonlocal Fisher--KPP equation.  
Namely, we consider bounded non-negative solutions $u=u(x)$ on the real line $\R$ to the following equation 
\begin{equation}
\label{eq:basic}
	\kap (a^+\ast u)(x) - mu(x) - \kam u(x) (a^-\ast u)(x)=0,\quad x\in\R,
\end{equation}
where $\kap,\kam$ and $m$ are (strictly) positive real numbers, $a^+$ and $a^-$ are probability densities, and the convolution terms are defined as follows
\[
	 (a^\pm\ast u)(x) = \int[\R] a^\pm(y) u(x-y)~\txtd y, \quad x\in\R.
\]
The evolution equation corresponding to \eqref{eq:basic} first appeared, for the case $\kap a^+ = \kam a^-$, $m=0$, in~\cite{Mol1972,Mol1972a}. For the case $\kap a^+ = \kam a^-$, $m>0$ we refer to~\cite{Dur1988} and for different kernels to~\cite{BP1997}, where the so-called Bolker--Pacala model of spatial ecology was considered. The equation~\eqref{eq:basic} was rigorously derived from the Bolker--Pacala model in~\cite{FM2004} for integrable $u$ and in~\cite{FKK2011} for bounded $u$. The long-time behavior was studied in \cite{FKT2015,FKT2016,FT2017,FT2017c}. In~\cite{AGV2006}, the term $\kap(a\ast u-u)$ was (formally) approximated by the Laplace operator using the Taylor expansion of the convolution term; see Section~\ref{sec:relation_to_FKPP_with_nonloc_reaction} below for more detail. In this approximation limit, one obtains the Fisher--KPP equation with a non-local reaction, 
\begin{equation}
\label{eq:F-KPP_nonloc_reaction}
	\partial_t u(x) = d\, \partial^2_{x} u(x) + \kam u(x)(\theta - (a^-\ast u)(x)), \quad x\in\R,
\end{equation}
where $d:= \frac{\kap}{2}\int[\R] y^2 a^+(y)~\txtd y$, $\theta$ is given in \eqref{eq:constant_solutions}, and \eqref{eq:F-KPP_nonloc_reaction} also often referred to as the nonlocal Fisher-KPP equation. Observe that there are two constant solutions to~\eqref{eq:basic} and~\eqref{eq:F-KPP_nonloc_reaction}, namely,
\begin{equation}
\label{eq:constant_solutions}
	u\equiv 0, \qquad u\equiv  \theta := \frac{\kap-m}{\kam}.
\end{equation}
It was pointed out in~\cite{Bri1989} that under additional assumptions the nonlocal Fisher-KPP equation~\eqref{eq:F-KPP_nonloc_reaction} admits a steady state bifurcation of $u\equiv\theta$ leading to existence of spatially periodic solutions. Later, more detailed analysis was carried out for a more general reaction in \cite{Bri1990}. Numerical analysis of bifurcations and traveling waves to \eqref{eq:F-KPP_nonloc_reaction} was considered in \cite{ABVV2010,AGV2006,Gou2000,GCD2001,FKK2003}. Analytical results for stationary solutions and traveling waves to~\eqref{eq:F-KPP_nonloc_reaction} can be found in~\cite{AC2012,AK2015,BNPR2009,FH2015,HR2014}. We also remark that there is a variant of the Fisher-KPP equation with a nonlocal operator replacing the Laplacian and with a local reaction~\cite{CD2007,G2011}.

In this paper we demonstrate that under additional assumptions there exists a steady-state bifurcation of $u\equiv\theta$ for the doubly-nonlocal Fisher-KPP equation~\eqref{eq:basic}. This bifurcation leads to existence of periodic solutions to~\eqref{eq:basic} and is connected with results to~\eqref{eq:F-KPP_nonloc_reaction} (specifically \cite{FH2015}) as we show in Section~\ref{sec:relation_to_FKPP_with_nonloc_reaction} via the singular local diffusion limit. Up to our knowledge, in contrast to \eqref{eq:F-KPP_nonloc_reaction}, the only results on bifurcations in \eqref{eq:basic} were done heuristically in recent publications \cite{Ayd2018,Bar2017}. Thus we present the first rigorous statements of this sort.

We stress that the problem of existence of stationary solutions to the equation~\eqref{eq:basic} depends on relations between parameters of the equation~\eqref{eq:basic}. In particular, if $\kap<m$, then $u\equiv0$ is the only non-negative bounded solution to~\eqref{eq:basic}, which follows from the Duhamel formula ($\theta<0$ in this case). If $\kap>m$, $\kap a^+(x) \geq (\kap-m)a^-(x)$, for $x\in\R$, and $a^+,a^-$ are symmetric, then the constant solutions given by \eqref{eq:constant_solutions} are the only non-negative bounded solutions to \eqref{eq:basic} (see Proposition~\ref{prop:nonexistence_of_stationary_sol} below).
If $a^+,a^-$ are non-symmetric, it is possible that there exist a traveling wave with a speed $0$, namely, there can exist decreasing $u:\R\to[0,\theta]$ which satisfies \eqref{eq:basic} and such that $u(+\infty)=0$, $u(-\infty)=\theta$ (see \cite{FKT2015}). Therefore, we have to carefully investigate, under which conditions bifurcations are possible. In Section~\ref{sec:assumptions} we formulate assumptions sufficient for a steady-state bifurcation of $u\equiv\theta$. First, we introduce a small parameter $\eps$ in \eqref{eq:basic} substituting $\kap,\kam$ by 
\[
\kape = (1+\eps)\kap \qquad \text{and}\qquad \kam_\eps = \left(1+ \eps\frac{\kap}{\kap{-}m}\right) \kam 
\]
correspondingly, which turns out to be a more suitable compact notation to state our results. Studying the problem for symmetric $a^\pm$ in the space of square-integrable periodic function on the real line, we show that the spectrum of the linearization of the left-hand side of \eqref{eq:basic} at $u\equiv\theta$ equals to the following set,
\[
	\{-\kape\} \cup \{ \alpha(\eps,k):=\kape \widehat{a^+}(k) - (\kape-m) \widehat{a^-}(k) - \kape \,\vert\, k\in \R\} \subset \R,
\]
where $\widehat{a^\pm}$ is the Fourier transform of $a^\pm$ defined below in~\eqref{def:Fourier-L1}. Next, we require, that for small $\eps<0$ the spectrum belongs to the negative half-plane $\{z\in\C\,\vert\,\Re z<0\}$, it touches the imaginary axis $\{z\in\C\,\vert\, \Re(z) = 0\}$ for $\eps=0$, and it intersects the positive half-plane $\{z\in\C\,\vert\, \Re(z)>0\}$ for small $\eps>0$. Thus, we have the following picture:  

\includegraphics[width=\textwidth]{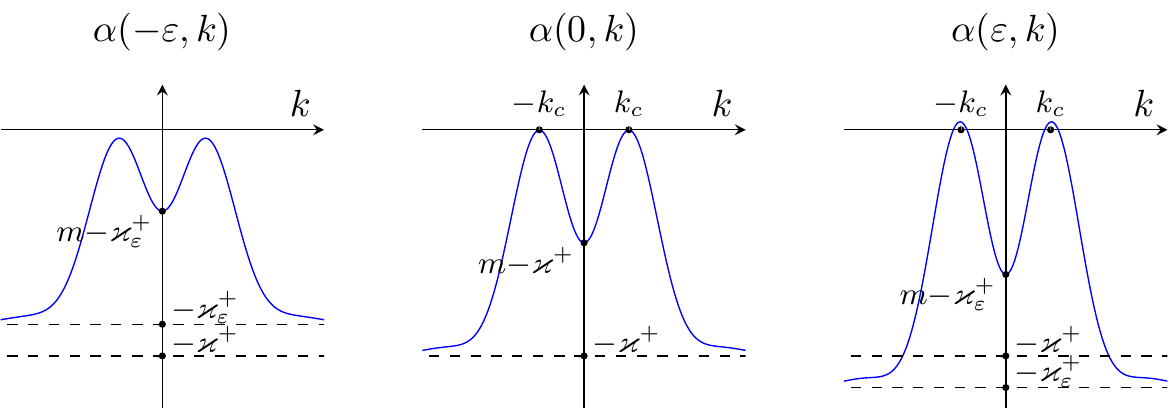}

These assumptions impose constraints on the parameters, yet they imply existence of periodic solutions to \eqref{eq:basic}, that we prove in Section~\ref{sec:existence} applying the Lyapunov--Schmidt reduction method. Namely, we demonstrate that for any sufficiently small $\eps>0$ and $\delta$ (probably negative), there exists a periodic solution to \eqref{eq:basic} with a period $\frac{2\pi}{k_c+\delta}$, where $k_c>0$ is such that $\alpha(0,k_c)=0$.

\begin{figure}[H]
	\includegraphics[width=\textwidth]{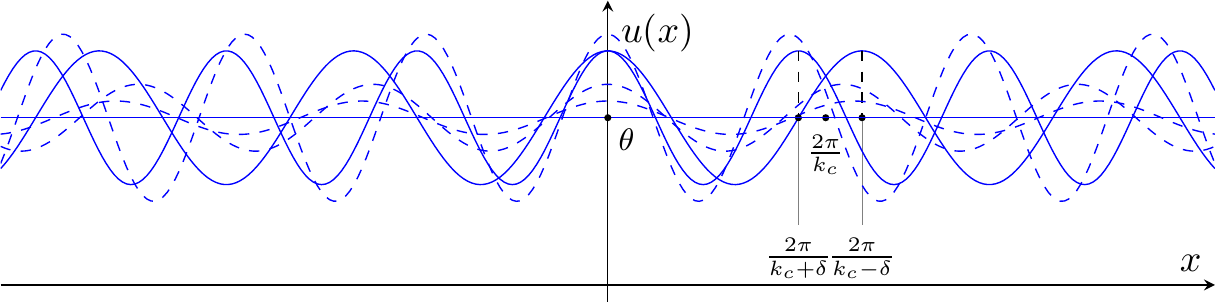}
	\caption{\text{Sketch of the existence of classes of periodic solutions for small }$\eps$\text{ and }$\delta$.}
\end{figure}

Section~\ref{sec:stability} is devoted to the study of stability of the solutions in a space of square-integrable periodic functions. We show, that the solutions are (locally) asymptotically stable with respect to perturbations with the same period and phase. Section~\ref{sec:examples} provides several examples of the probability densities $a^+$, $a^-$, which satisfy assumptions of the previous sections. Since our article is related to the results for~\eqref{eq:F-KPP_nonloc_reaction}, specifically to~\cite{FH2015}, we demonstrate this detailed connection in a certain limiting case in Section~\ref{sec:relation_to_FKPP_with_nonloc_reaction}. Section~\ref{sec:nonexistence} presents some results on non-existence of solutions to \eqref{eq:basic}, which shows that one really has to distinguish fundamentally different behaviour already for the stationary solutions of doubly-nonlocal equation.

\section{Assumptions} 
\label{sec:assumptions}

We start with some notation and background results. We denote $\R_+:=[0,\infty)$, $\|\cdot\|_p:=\|\cdot\|_{L^p(\R)}$, for $p\in[1,\infty]$. We will use Bachmann--Landau big $\cO$ and little $o$ notations
\[
	f(x) = \cO(g(x)),\ x\to x_0, \qquad f(x) = o(g(x)),\ x\to x_0.
\]

\begin{lem}\label{lem:lone_linft}
 Let $a\in L^1(\X)$ and $f\in L^\infty(\X)$. Then it follows that $a\ast f\in BUC(\X)$.
\end{lem}
\begin{proof}
See e.g.  \cite[Lemma 3.1]{FT2017}
\end{proof}
\begin{cor}\label{cor:solution_is_buc}
	If $u\in L^\infty(\R\to\R_+)$ solves \eqref{eq:basic}, then $u\in \Buc$.
\end{cor}
\begin{proof}
	If $u\in L^\infty(\R\to\R_+)$ solves \eqref{eq:basic}, then
	\[
		0 \leq u(x) = \frac{\kap (a^+\ast u)(x)}{\kam (a^-\ast u)(x) + m} \leq \frac{\kap \|u\|_\infty}{m}.
	\]
	Hence, the statement follows from Lemma~\ref{lem:lone_linft}.
\end{proof}

Our first main assumptions are the following
\begin{assum}\label{assum:assumptions_basic}
	\begin{gathered}\kap{>}m, \qquad a^\pm(-x)\equiv a^\pm(x),\\  \int[\R] x^2 a^\pm(x)~\txtd x<\infty, \qquad  a^\pm(x) \leq \frac{C}{1+|x|^{1{+}\xi}},\ x\in\R,
	\end{gathered}
\end{assum}
where $C,\xi>0$ are some fixed constants. Note that the first assumption in~\eqref{assum:assumptions_basic} already hints at the fact that we must impose certain growth restrictions on the linear part to obtain bifurcation results. The further assumptions are typical technical assumptions on the kernel(s) for nonlocal and doubly-nonlocal Fisher-KPP equations. We introduce a small parameter $\eps$, to study structural changes of the solutions to \eqref{eq:basic} under small perturbations of this parameter. To simplify our notations, we will write $\kape :=  (1+\eps)\kap$. Let us suppose that the constant solution $u\equiv \theta = \frac{\kap-m}{\kam}$ is independent of $\eps$ and so we have
\[
\frac{\kap_\eps - m_\eps}{\kam_\eps} = \theta,
\]
where subscripts denote new parameters. Let us also assume that $m$ is not changed so that $m_\eps \equiv m>0$. 
As a result of the parameter change, the coefficients in~\eqref{eq:basic} are transformed as follows, 
\begin{equation}
\label{eq:def_kape_kame_me}
	\kape = (1+\eps)\kap, \quad \kam_\eps = \left(1+ \eps\frac{\kap}{\kap{-}m}\right) \kam, \quad m_\eps = m.
\end{equation}
If we set $w := u - \theta$, then $w$ satisfies the following equation
\begin{equation}\label{eq:basic_shifted}
  \kape a^+\ast w - (\kape{-}m)a^-\ast w - \kape w - \kame w a^-\ast w = 0, \quad x\in\R.
\end{equation}
We will study bifurcations of $u\equiv \theta$ in the class of periodic functions, i.e., bifurcations of $w$ from the branch of trivial solutions a $w\equiv0$. Therefore, we introduce the following (complex) Hilbert space of periodic square-integrable functions with a period $p>0$,
\begin{gather*}
   L^2_p(\R) := \{ f:\R\to\C \vert f\in L^2([0,p]);\ f(x) = f(x+p),\ x\in\R\}, \label{def:L2p} \\
	(f,g)_{L^2_p} := \frac{1}{p} \int[0][p] f(x)\bar{g}(x)~\txtd x. 
\end{gather*}
Let us ensure that $w\in L^2_p(\R)$ implies $w(a^-\ast w)\in L^2_p(\R)$. 
\begin{prop}\label{prop:w_a_w_in_l2p}
	Let $w \in L^2_p(\R)$ and $a\in L^1(\R\to\R_+)$ be such that
	\begin{equation*}
		I_p(a):= \sqrt{p} \sup_{x\in\R} \sum_{j\in\Z} \|a(x{-}\cdot)\|_{L^2([jp,(j+1)p))} < \infty.
	\end{equation*}
	Then
	\begin{equation}\label{eq:l2p_linf_est}
		\frac{\|w (a\ast w) \|_{L^2_p(\R)}}{\|w\|_{L^2_p(\R)}} \leq \|a\ast w\|_\infty \leq I_p(a) \|w\|_{L^2_p(\R)}.
	\end{equation}
	In particular, $w\in L^2_p(\R)$ implies $w(a^-\ast w)\in L^2_p(\R)$ if for some $C,\xi>0$, 
  \[
  	a(x) \leq \frac{C}{1+|x|^{1+\xi}}, \quad x\in\R.
  \]
\end{prop}

\begin{proof}
	Since $w$ is $p-$periodic, we have
	\begin{align*}
		\frac{\|w (a\ast w) \|_{L^2_p(\R)}}{\|w\|_{L^2_p(\R)}} &\leq \sup_{x\in\R} |(a\ast w)(x)| \leq \sup_{x\in\R} \sum_{j\in\Z} \int[jp][(j+1)p] a(x-y) |w(y)| ~\txtd y \\
		&\leq \sup_{x\in\R} \sum_{j\in\Z} \sqrt{p} \|a(x{-}\cdot)\|_{L^2([jp,(j+1)p) )} \|w\|_{L^2_p(\R)}\\ 
		&=  I_p(a)\|w\|_{L_2^p(\R)}.
	\end{align*}
	The last statement of the proposition follows from  the estimate
	\[
		I_p(a) \leq \sum_{j\in\N}\frac{2C\sqrt{p}}{1+|jp|^{1+\xi}} <\infty,
	\]
	which finishes the proof.
\end{proof}

\begin{cor}
\label{cor:lp_is_linfity}
	Consider $p>0$ and suppose that $u\in L_p^2(\R\to\R_+)$ satisfies \eqref{eq:basic} and $I_p(a^+)<\infty$.
	Then $u\in L^\infty(\R)$.
\end{cor}
\begin{proof} By \eqref{eq:l2p_linf_est}, the proof follows from the following estimate
\[
	u(x) = \frac{\kap (a^+\ast u)(x)}{m+ \kam (a^-\ast u)(x)} \leq \frac{\kap}{m} I_p(a^+) \|u\|_{L^2_p(\R)}.\qedhere 
\]
\end{proof}

As a next step, it is helpful to introduce the wave number $k$. For $a\in L^1(\R)$ and $f\in L^2_{\frac{2\pi}{k}}(\R)$ observe that
\[
	(a\ast f)(\tfrac{x}{k}) = (a_k\ast f^k)(x), \quad x\in\R,
\]
where $a_k(x) := \frac{1}{k} a(\frac{x}{k})$, $f^k(x) := f(\tfrac{x}{k}) \in L^2_{2\pi}(\R)$. Hence, if $w$ satisfies \eqref{eq:basic_shifted} and $v(x) = w(\tfrac{x}{k})$, then $v$ satisfies the following equation, for $x\in\R$,
\begin{equation}
\label{eq:basic_shifted_scaled}
	F(v,\eps, k) := \underbrace{\kape a^+_{k}\ast v - (\kape{-}m)a^-_{k}\ast v - \kape v}_{=: \Aek v } \underbrace{- \kame v a^-_{k}\ast v}_{=:\mathit{R}(v,\eps,k)}=0, 
\end{equation}
which is the main bifurcation problem we study near the point $(v,\eps)=(0,0)$. In particular, instead of considering \eqref{eq:basic_shifted} with $w\in \cup_{k>0} L^2_{\frac{2\pi}{k}}(\R)$, we consider \eqref{eq:basic_shifted_scaled} on $L^2_{2\pi}(\R)$, passing to the space with the fixed period $2\pi$.

\begin{rem}
Since, by \eqref{assum:assumptions_basic}, $a^\pm(x) \equiv a^\pm(-x)$, we can conclude that $\Aek$ given by \eqref{eq:basic_shifted_scaled} is a bounded self-adjoint operator in $L^2_{2\pi}(\R)$. Therefore, the spectrum of $\Aek$, denoted by $\sigma(\Aek) \subset \R$, is a bounded set.
\end{rem}

To observe a bifurcation at $\eps=0$, we assume that the spectrum of the operator $\mathcal{A}_{\eps,k}$ passes through the imaginary axis at $\eps=0$, and some $k=k_c>0$, namely, there exist $k_c,\eps_0,\delta_0>0$, such that for all $\eps\in(0,\eps_0)$, $\delta\in(-\delta_0,\delta_0)$, in $L^2_{2\pi}(\R)$,  
\begin{gather}
	\sigma(\mathcal{A}_{-\eps,k_c+\delta}) \subset (-\infty,0); \ \ \{0\} \in \sigma(\mathcal{A}_{0,k_c}); \ \  \sigma(\mathcal{A}_{\eps,k_c+\delta})\cap (0,\infty) \neq \varnothing.  \label{assum:spectrum_intersects_zero}
\end{gather}
It is helpful to re-interpret the last assumption more concretely in Fourier space. Consider the Fourier transform of $f\in L^2_{2\pi}(\R)$ defined by
\begin{equation*}
	(\F f)(j) := \frac{1}{2\pi} \int[0][2\pi] \e^{-\i jx} f(x) ~\txtd x, \qquad j\in\Z,
\end{equation*}
and the Fourier transform of $a\in L^1(\R)$ given by
\begin{equation}\label{def:Fourier-L1}
	\widehat{a}(k) := \int[\R] a(x)\e^{-\i k\cdot x} ~\txtd x, \quad k\in\R.
\end{equation}
Obviously, if $a\in L^1(\R)$ and $f\in L^2_{2\pi}(\R)$, then $a\ast f\in L^2_{2\pi}(\R)$, and
\[
	(\F (a\ast f))(j) = \widehat{a}(j) (\F f)(j), \quad j\in\Z.
\]
We denote
\[
	l_2 := \left\{ a=\{a_j\in\C\}_{j\in\Z} : \sum_{j\in\Z} |a_j|^2 <\infty \right\}, \qquad (a,b)_{l_2} := \sum_{j\in\Z} a_j \bar{b}_j.
\]
By the Plancherel formula $\F: L^2_{2\pi}\to l_2$ is a unitary operator, which implies
\[
	\sigma(\F \Aek \F^{-1}) = \sigma(\Aek).
\]
Since, for $a_k(x) := \frac{1}{k} a(\frac{x}{k})$, $ \hat{a}_k(j) = \hat{a}(kj)$, it follows we get for $f\in L^2_{2\pi}(\R)$,
\begin{gather}
	(\F \Aek f)(j) = \alpha(\eps,jk) (\F f)(jk), \quad j\in\Z, \nonumber\\ 
	\text{where} \quad \alpha(\eps,p) := \kape \widehat{a^+}(p) - (\kape-m) \widehat{a^-}(p) - \kape. \label{eq:def_alpha}
\end{gather}
Therefore, $\sigma(\Aek)$ is the closure of the set $\{ \alpha(\eps, jk) | j\in\Z\}$. Since $\alpha(\eps,p) \to -\kape$ as $|p|\to\infty$, the condition \eqref{assum:spectrum_intersects_zero} follows from the assumption
\begin{equation}\label{eq:alpha_eps_delta_less_alpha_0_0}
	\alpha(-\eps, k_c+\delta) < \alpha(0, k_c)=0 < \alpha(\eps,k_c+\delta), \quad \eps\in(0, \eps_0),\ |\delta|<\delta_0.
\end{equation}
We want to re-formulate the last bifurcation condition more concisely in terms of $\alpha$ and its derivatives. First, we have to assume that
\begin{assum}
\label{assum:spectrum_intersects_zero_ii}
	\exists k_c>0:\quad \alpha(0,k_c)=0.
\end{assum}
We also assume that there are well-distinguished critical modes
\begin{assum}\label{assum:alpha_doesnt_have_other_zeroes}
	\alpha(0,j k_c) \neq 0, \quad j\neq\pm1.
\end{assum}
Note that $\Re(\alpha(0,p)) = \alpha(0,p) = \alpha(0,-p)$, since $a^\pm(x) \equiv a^\pm(-x)$. By \eqref{assum:assumptions_basic} and \eqref{eq:def_alpha}, we have,
\begin{equation}\label{eq:alpha_is_smooth}
	\exists \rho>0: \quad (\eps,k)\to\alpha(\eps,k) \in C^2((-\rho,\rho)\times(k_c-\rho,k_c+\rho)).
\end{equation}
In order for \eqref{eq:alpha_eps_delta_less_alpha_0_0} to hold, by \eqref{eq:alpha_is_smooth}, we assume, 
\begin{assum}\label{eq:dalpha_dk_zero}
	\frac{\partial}{\partial k} \alpha(0,k_c) = 0 \qquad \Big( \Leftrightarrow\ \kap\frac{\partial}{\partial k} \widehat{a^+}(k_c) = (\kap{-}m)\frac{\partial}{\partial k} \widehat{a^-}(k_c) \Big).
\end{assum}
We can understand the spectrum near the critical wave number by considering the following expansion
\begin{align}
	\alpha(\eps, k_c+\delta) =& \underbrace{\alpha(0,k_c)}_{=0\text{ by }\eqref{assum:spectrum_intersects_zero_ii}} + \underbrace{\diff{}{k}\alpha(0,k_c)\delta}_{=0\text{ by }\eqref{eq:dalpha_dk_zero}} + \underbrace{\diff{}{\eps} \alpha(0,k_c) \eps + \frac{1}{2}\frac{\partial^2}{\partial k^2} \alpha(0,k_c) \delta^2}_{=:\,\Omega(\eps,\delta)} \nonumber \\ 
	&+ \underbrace{\frac{\partial^2}{\partial k\partial \eps} \alpha(0,k_c) \delta \eps}_{=o(|\eps|+\delta^2),\ \delta\to 0,\,\eps\to 0} + \underbrace{\frac{1}{2}\frac{\partial^2}{\partial \eps^2}  \alpha(0,k_c) \eps^2}_{=0\text{ by the def. of }\alpha} + o(\delta^2 + \eps) \nonumber\\
		=&\ \Omega(\eps,\delta) +  o(\delta^2+\eps),\quad \delta\to0,\ \eps\to0. \label{eq:alpha_series_decomposition}
\end{align}
By \eqref{assum:spectrum_intersects_zero_ii}, we obtain
\[
	\frac{\kap-m}{\kap} \widehat{a^-}(k_c) = \widehat{a^+}(k_c) -1 <0.
\]
Hence, we automatically get a transversality condition for bifurcation parameter
\begin{equation}\label{eq:dalpha_deps}
	\diff{}{\eps} \alpha(0,k_c) = \kap(\widehat{a^+}(k_c) - \widehat{a^-}(k_c) -1) = - m\widehat{a^-}(k_c)  > 0.
\end{equation}
In order to satisfy the first inequality in \eqref{eq:alpha_eps_delta_less_alpha_0_0} (consider e.g. $-\eps = \delta^3<0$ in \eqref{eq:alpha_series_decomposition}), we assume 
\begin{assum}\label{assum:d2_alpha_dk2_less_zero}
	\frac{\partial^2}{\partial k^2} \alpha(0,k_c)<0 \qquad \Big( \Leftrightarrow \ \kap\frac{\partial^2}{\partial k^2} \widehat{a^+}(k_c) < (\kap{-}m)\frac{\partial^2}{\partial k^2} \widehat{a^-}(k_c)  \Big).
\end{assum}
To ensure the second inequality in \eqref{eq:alpha_eps_delta_less_alpha_0_0} it is sufficient to suppose that $\Omega(\eps,\delta)>0$, which, by \eqref{eq:dalpha_deps} and \eqref{assum:d2_alpha_dk2_less_zero}, holds if 
\begin{assum}\label{assum:delta_less_eps}
	\frac{\delta^2}{\eps} < \frac{2\frac{\partial}{\partial \eps} \alpha(0,k_c)}{-\frac{\partial^2}{\partial k^2} \alpha(0,k_c)} = \frac{2m\widehat{a^-}(k_c)}{\frac{\partial^2}{\partial k^2} \alpha(0,k_c)}.
\end{assum}
We denote $\Ac = \mathcal{A}_{0,k_c}$. Now we can check that our assumptions limit the critical modes to a two-dimensional space:  

\begin{lem}\label{lem:spectrum_Ac_is_2d}
  Let \eqref{assum:assumptions_basic}, \eqref{assum:spectrum_intersects_zero_ii} and \eqref{assum:alpha_doesnt_have_other_zeroes} hold, then	
	\begin{equation}\label{eq:spectrum_Ac_is_2d}
		\ker  \Ac = {\textnormal{span}}\{\e^{\i x},\e^{-\i x}\}.
	\end{equation}
\end{lem}

\begin{proof}
By \eqref{assum:assumptions_basic} and \eqref{assum:spectrum_intersects_zero_ii} one easily concludes $\{\alpha(0,\pm k_c) = 0\} \in \sigma(\Ac)$. Moreover, the following equalities hold
\[
	a\ast \e^{\pm \i k_c x} = \e^{\pm \i k_c x} \hat{a}(\pm k_c), \qquad a_k\ast \e^{\pm \i x} = a\ast \e^{\pm \i kx}.
\]
Hence, we obtain $\Ac \e^{\pm \i x} = \e^{\pm \i x} \alpha(0, \pm k_c) = 0$. Thus $\e^{\i x},\e^{-\i x}$ are eigenvectors for the eigenvalue $\la=0$. If $f\in \ker \Ac$, then we find
\[
	0 = \Ac f = \Ac \sum_{j\in\Z} (f,\e^{\i jx}) \e^{\i jx} = \sum_{j\in\Z} \alpha(0,jk_c) (f,\e^{\i jx}) \e^{\i jx}.
\]
So our assumption \eqref{assum:alpha_doesnt_have_other_zeroes} yields $(f,\e^{\i jx})=0$ for $j\neq \pm1$.
Thus $f\in {\textnormal{span}}\{\e^{\i x},\e^{-\i x}\}$ follows. As a result, \eqref{eq:spectrum_Ac_is_2d} holds, which finishes the proof.
\end{proof}

In order to prove existence of non-constant solutions to \eqref{eq:basic} we also need the assumption
\begin{assum}\label{assum:omega_non_zero}
	\omega:=  (\kam)^2 \widehat{a^-}(k_c) \Big( \frac{\widehat{a^-}(k_c)+\widehat{a^-}(2k_c)}{\alpha(0,2k_c)} + \frac{2+2\widehat{a^-}(k_c)}{\alpha(0,0)} \Big) > 0.
\end{assum}
This last assumption is not yet transparent but it will be interpreted below as a local solvability condition to obtain a real 
branch of non-trivial solutions; see equation~\eqref{eq:periodic_solution}.

\begin{rem}
\label{rem:kam_unimportant}
Note that \eqref{assum:assumptions_basic}-\eqref{assum:omega_non_zero} are independent of $\kam$, which is natural, since the linearization $\Aek$ does not depend on $\kam$ (c.f. \eqref{eq:basic_shifted_scaled}).
\end{rem}


\section{Existence of periodic solutions} \label{sec:existence}

The following theorem states that under \eqref{assum:assumptions_basic}--\eqref{assum:omega_non_zero} there exist a steady-state bifurcation of $u\equiv\theta$:
 
\begin{thm}
\label{thm:steady_state_bifurcation}
  Let \eqref{assum:assumptions_basic}--\eqref{assum:omega_non_zero} hold. Then there exists $\eps_0>0$, such that for all $\eps\in(0,\eps_0]$ and all $\delta$ that satisfy \eqref{assum:delta_less_eps} there exists a $\frac{2\pi}{k_c+\delta}$-periodic solution to \eqref{eq:basic} (where $\kap$, $\kam$ are replaced by $\kape$, $\kame$) with the leading expansion of the form
  \begin{equation}
	\label{eq:periodic_solution}
	u_{\eps,\delta}(x) = \theta + 2\sqrt{\frac{ \Omega(\eps,\delta)}{\omega}} \cos((k_c+\delta)x) + o(|\eps|^{\frac12} + |\delta|), \quad \delta\to0,\ \eps\to0,
  \end{equation}
	where $\Omega$ is defined in \eqref{eq:alpha_series_decomposition}, $\omega$ is defined in \eqref{assum:omega_non_zero}, and $k_c$ satisfies \eqref{assum:spectrum_intersects_zero_ii}-\eqref{assum:omega_non_zero}.
\end{thm}

We will follow a similar strategy as in \cite{Kie2012} and apply the Lyapunov--Schmidt reduction method in order to give a proof of Theorem~\ref{thm:steady_state_bifurcation}. 

\begin{proof}
We study bifurcation from the trivial solution branch $\{(0,\eps,k)\, \vert\, \eps, k\in\R\}$ for 
\[
	F(v,\eps,k)=0, \quad  v\in X=L_{2\pi}^2(\R),
\]
where $F$ is given by \eqref{eq:basic_shifted_scaled}. Under \eqref{assum:assumptions_basic}--\eqref{assum:delta_less_eps}, consider $\Ac = \txtD_vF(0,0,k_c):X\to X$. Then we clearly have the splitting
\begin{equation}
\label{eq:space_decomposition}
	X = L^2_{2\pi}(\R) = \ker (\Ac) \oplus {\textnormal{ran}}(\Ac) = {\textnormal{span}}\{\e^{\i x}, \e^{-\i x}\} \oplus {\textnormal{ran}}(\Ac). 
\end{equation}
Hence $F(v,\eps,k)=0$ is equivalent to  
\begin{align}
	P F(y + \psi, \eps, k_c+\delta) &= 0, \quad y = Pv,\ \psi = (1-P)v, \label{eq:reduced} \\
 (1-P)F(y+\psi,\eps, k_c+\delta) &=0	, \label{eq:implicit}
\end{align}
where $k=k_c+\delta$ and $P$ is the projection on $\ker (\Ac)$. The idea of the proof is to find $\psi=\psi(y,\eps,\delta)$, which satisfies \eqref{eq:implicit}, then put it in \eqref{eq:reduced} and find $y=y(\eps,\delta)$, which satisfies \eqref{eq:reduced}. Finally, we are going to obtain that 
\[
v(\eps,\delta) = y(\eps,\delta) + \psi(y(\eps,\delta),\eps,\delta) 
\]
will be a solution to \eqref{eq:basic_shifted_scaled}. By~\eqref{assum:assumptions_basic} we have a twice-differentiable mapping 
\[
(v,\eps,k) \mapsto F(v,\eps,k),\qquad F\in C^{2}(X\times(-1,\infty)\times(0,\infty)\to \R). 
\]
Since we can just calculate
\[
	\txtD_\psi((1-P)F)(0,0,k_c) = (1-P)\txtD_\psi F(0,0,k_c) = (1-P)\Ac,
\]
and $(1-P)\Ac:{\textnormal{ran}}(\Ac)\to {\textnormal{ran}}(\Ac)$ is a linear homeomorphism, we may apply the Implicit Function Theorem locally. More precisely, there exist an open $U\subset \ker (\Ac)$ with $\{0\}\in U$, $\delta_0>0$, $\eps_0>0$, and 
\begin{equation}
\label{eq:psi_smooth}
	\psi=\psi(y,\eps,\delta) \in C^{2}(U\times(-\eps_0,\eps_0)\times(-\delta_0,\delta_0)\to {\textnormal{ran}}(\Ac)),
\end{equation}
such that \eqref{eq:implicit} holds and $\psi(0,0,0)=0$. Note that, upon possibly redefining $\eps_0,\delta_0$, we get from~\eqref{assum:alpha_doesnt_have_other_zeroes} and $\alpha(\eps,\pm\infty)=-\kape<0$ that
\begin{equation}\label{eq:perturbed_A3}
	\alpha(\eps,j(k_c+\delta)) \neq 0, \quad \eps\in(-\eps_0,\eps_0),\ \delta\in(-\delta_0,\delta_0),\ \Z\ni j\neq \pm1.
\end{equation}
Therefore, $\Aek \e^{\i jx} = \alpha(\eps,j(k_c+\delta)) \e^{\i jx} \neq 0$ for $k=k_c+\delta$ and $j\neq\pm1$, which implies that $\Aek$ is a linear diffeomorphism on ${\textnormal{ran}}(\Ac)$ and 
\begin{equation}\label{eq:Aek_inverse}
	(\Aek^{-1} f)(x) = \sum_{\Z \ni j \neq \pm1} \frac{f_j}{\alpha(\eps,jk)} \e^{\i jx},\ \  f(x)=\sum\limits_{\Z\ni j\neq \pm1} f_j\e^{\i jx} \in {\textnormal{ran}}(\Ac).  
\end{equation}
Next, by \eqref{eq:psi_smooth}, $\psi$ satisfies the following expansion in $y$ for $|\eps|<\eps_0$, $|\delta|<\delta_0$,
\begin{equation}\label{eq:periodic_solution_expansion}
	\psi(y,\eps,\delta) = \sum_{l=0}^{2} G_{l}(y,\eps,\delta) + o(\|y\|^2), \quad \|y\| \to 0, 
\end{equation}
where $G_{l}:\ker (\Ac)^l\times(-\eps_0,\eps_0)\times(-\delta_0,\delta_0) \to {\textnormal{ran}}(\Ac)$ are $l$-linear forms with respect to the first argument and belong to the class $C^2$ with respect to $(\eps,\delta)$. We have to compute $G_{l}$ for $l=0,1,2$ as defined in \eqref{eq:periodic_solution_expansion}. Collecting zero-forms with respect to $y$ in \eqref{eq:implicit} we obtain
\begin{equation}\label{eq:G_0}
	\mathcal{A}_{\eps,k_c+\delta} G_0(\eps,\delta) = \kame (1-P)G_0(\eps,\delta) a^-_{k_c+\delta}\ast G_0(\eps,\delta).
\end{equation}
Since $\psi(0,0,0)=0$ and \eqref{eq:psi_smooth} holds we see that 
\[
	G_0(0,0)=0, \quad G_0\in C^2((-\eps_0,\eps_0)\times(-\delta_0,\delta_0)).
\]
Therefore, by \eqref{eq:Aek_inverse} and the Implicit Function Theorem, there exists a unique $C^2$-solution to \eqref{eq:G_0} in ${\textnormal{ran}}(\Ac)$, which equals $0$ at $\eps=0, \delta=0$. Obviously $G_0\equiv0$ satisfies these conditions so we proceed to the next order, i.e., we collect one-forms with respect to $y$ in \eqref{eq:implicit}. A direct calculation yields
\[
 (1-P)\mathcal{A}_{\eps,k_c+\delta} G_1(y,\eps,\delta) = 0, \quad y\in \ker (\Ac).
\]
Since ${\textnormal{ran}}(G_1)\subset {\textnormal{ran}}(\Ac)$, then by \eqref{eq:Aek_inverse} it also follows that $G_1 \equiv 0$. Going to the next order, we collect two-forms in \eqref{eq:implicit} and obtain
\begin{equation*}
	(1-P) \mathcal{A}_{\eps,k_c+\delta} G_2(y,\eps,\delta) = (1-P)\kame y(a^-_{k_c+\delta}\ast  y), \quad y\in \ker (\Ac).
\end{equation*}
Since, for all $y\in \ker (\Ac)$, $y a^-_{k_c+\delta}\ast y \in {\textnormal{ran}}(\Ac)$, ${\textnormal{ran}}(G_2) \subset {\textnormal{ran}}(\Ac)$ and \eqref{eq:Aek_inverse} hold, we get  
\begin{equation}\label{eq:G2}
	G_2(y,\eps,\delta) = \kame \mathcal{A}_{\eps,k_c+\delta}^{-1} [y(a^-_{k_c+\delta}\ast  y)], \quad y\in \ker (\Ac).
\end{equation}
We are looking for real-valued stationary solutions to \eqref{eq:basic_shifted_scaled}. Therefore, it is straightforward to check that if $v=y+\psi$ is real-valued, then $y$ and $\psi$ are real-valued and $y=s\e^{\i x}+\bar{s}\e^{-\i x}$ for some $s\in\C$; see also Remark~\ref{rem:real_valued_functions} below. Next, for $y=s \e^{\i x} +\bar{s} \e^{-\i x}$ we can actually write out the nonlinear quadratic term
\begin{equation}\label{eq:G2_rhs}
	ya^-_{k_c+\delta}\ast  y = \widehat{a^-}(k_c+\delta) (s^2 \e^{2\i x} + 2|s|^2 + \bar{s}^2 \e^{-2\i x}).
\end{equation}
We have by \eqref{eq:Aek_inverse}, \eqref{eq:G2} and \eqref{eq:G2_rhs}, that $(G_2,\e^{\i jx}) = 0$ for $j\not\in \{\pm2, 0\}$, and 
furthermore
\[
G_2 = g_0 |s|^2 + g_2 (s^2 \e^{2\i x} + \bar{s}^2 \e^{-2\i x}), 
\]
where the coefficients are given by
\[
	g_2(\eps,\delta) = \frac{\kame \widehat{a^-} (k_c+\delta)}{\alpha(\eps,2(k_c+\delta))}, \quad g_0(\eps,\delta) = \frac{2\kame \widehat{a^-} (k_c+\delta)}{\alpha(\eps,0)}. 
\]
Hence, by \eqref{eq:periodic_solution_expansion}, any real-valued solution $v=y+\psi$ to \eqref{eq:basic_shifted_scaled} has the following form
\begin{equation}\label{eq:very_explicite_v}
	v(x) = s\e^{\i x} + \bar{s}\e^{-\i x} + g_2 s^2 \e^{2\i x} + g_0|s|^2 + g_{2}\bar{s}^{2} \e^{-2\i x} + o(|s|^2),\quad |s|\to0.
\end{equation}
Note that, by \eqref{eq:very_explicite_v}, $P \mathcal{A}_{\eps,k_c+\delta} v = \alpha(\eps,k_c+\delta)(s\e^{\i x}+\bar{s}\e^{-\i x})$ and, as $|s|\to0$,
\[
	Pva_{k_c+\delta}^-\ast v = [g_0(1+\widehat{a^-}(k_c+\delta)) + g_2 (\widehat{a^-}(k_c+\delta) + \widehat{a^-}(2(k_c+\delta)))] |s|^2 (s\e^{\i x}+\bar{s}\e^{-\i x}) + o(|s|^3). 
\]
Therefore, by \eqref{eq:alpha_series_decomposition} and since $\widehat{a^-}(k_c+\delta)= \widehat{a^-}(k_c)+\frac{1}{2}\partial^2_{k} \widehat{a^-}(k_c) \delta^2$ as $\delta\to0$, the substitution of $v$ given by \eqref{eq:very_explicite_v} into \eqref{eq:reduced} will imply the following reduced equation,
\begin{equation*}
	(\Omega(\eps,\delta)-\omega |s|^2)(s\e^{\i x}+\bar{s}\e^{-\i x}) = o(|s|^3+|\eps|+\delta^2), \quad |s|^3+|\eps|+|\delta|\to0,
\end{equation*}
where we used that 
\[
	|s|^3(\eps+\delta^2)=o(|s|^3+|\eps|+\delta^2), \quad |s|^3+|\eps|+\delta^2 \to 0,
\]
and the following asymptotic expansion, as $|\eps|+\delta^2\to0$,
\[
	\kame[g_0(1+\widehat{a^-}(k_c+\delta)) + g_2 (\widehat{a^-}(k_c+\delta) + \widehat{a^-}(2(k_c+\delta)))]\sim \omega + \cO (|\eps|+\delta^2).
\]
By \eqref{assum:omega_non_zero}, the Implicit Function Theorem implies for $|\eps| < \eps_0$ and $|\delta| < \delta_0$ (again possibly redefining $\eps_0$, $\delta_0$) that there exists $s=s(\eps,\delta)$ such that $v$ given by \eqref{eq:very_explicite_v} satisfies \eqref{eq:reduced}. As a result we obtain $v$ satisfies \eqref{eq:basic_shifted_scaled}), $s(0,0)=0$ and
\[
	 |s|^2 = \frac{\Omega(\eps,\delta)}{\omega} + o(|\eps|+\delta^2), \quad |\eps|+\delta^2 \to 0. 
\]
By \eqref{assum:delta_less_eps} and \eqref{assum:omega_non_zero} we also see that $|s|$ is real.
Let $s=|s|\e^{\i\phi}$ for some $\phi\in[0,2\pi)$.
	Then, by \eqref{eq:very_explicite_v} and \eqref{assum:omega_non_zero}, we can conclude that 
\[
	v(x) = 2|s| \cos(x+\phi) + o(|s|), \qquad |s|\to0.
\]
Since the semiflow which corresponds to \eqref{eq:basic} and the shift operator commute, then wlog we may assume $\phi=0$.  Hence \eqref{eq:periodic_solution} holds. 
\end{proof}

\begin{rem}\label{rem:real_valued_functions}
If  $\psi\in {\textnormal{ran}}(\Ac)$, then 
\[
	\bar{\psi}:= \Re( \psi) - i\Im( \psi) = \sum_{\Z\ni j\neq \pm 1} \overline{(\psi,\e^{\i jx})}\e^{-\i jx} = \sum_{\Z\ni j\neq \pm 1}(\e^{\i jx},\psi) \e^{-\i jx} \in {\textnormal{ran}} (\Ac).
\]
Hence, $\Re(\psi), \Im(\psi) \in {\textnormal{ran}}(\Ac)\perp \ker (\Ac)$.
Therefore $\Im(y+\psi) =0$, for $y\in \ker (\Ac)$, $\psi\in {\textnormal{ran}}(\Ac)$ implies $\Im(y) = \Im(\psi) =0$.
As a result, for $y=y(s,t) = s\e^{\i x} + t\e^{-\i x}$, $\Im(y)=0$ implies $s = \bar{t}$.
\end{rem}


\section{Stability of periodic solutions} \label{sec:stability}

Although we have shown the existence of non-trivial stationary solutions, we also would like to check whether one actually observe these solutions as long-time limit of the evolution problem. Therefore, we are going to study stability of the periodic solutions $u_{\eps,\delta}$ obtained in Theorem \ref{thm:steady_state_bifurcation} for small values of $\eps$ and $\delta$. Since $u_{\eps,\delta}(\cdot-h)$ solves \eqref{eq:basic} for any $h\in[0,2\pi)$ we consider $v_{\eps,\delta}(\cdot) = u_{\eps,\delta}(\frac{\cdot}{k_c+\delta})-\theta$ only in the following subspace of $L^2_{2\pi}(\R)$ (cf.~\eqref{eq:space_decomposition})
\begin{equation}\label{eq:space_of_perturbations}
	Y = \{l \cos(x):l\in\R\} \oplus {\textnormal{ran}} \Ac \subset L^2_{2\pi}(\R),
\end{equation}
where we reduce $\ker  \Ac = \{s \e^{\i x} + t \e^{-\i x}\}$ to its subspace with $t=\bar{s}$ and $arg(s)=0$. In other words $Y$ is chosen such that the phase of $u_{\eps,\delta}$ is fixed, namely, there exists unique $h\in[0,2\pi)$, such that $u_{\eps,\delta}(\cdot-h)\in Y$ (in fact $h=0$). The main result of the section is the following theorem.

\begin{thm}
\label{thm:lp_stability}
Let conditions of Theorem \ref{thm:steady_state_bifurcation} hold and let $u_{\eps,\delta}$ be the corresponding periodic solution to \eqref{eq:basic}. Suppose also 
\begin{equation}
\label{eq:ccstab}
	\lim_{\tilde{x}\to0} \int[\R] |a^\pm(\tilde{x}+x)-a^\pm(x)| ~\txtd x =0.
\end{equation}
Then there exists $\eps_0>0$ such that, for any fixed $\eps\in(0,\eps_0)$ and $\delta$ satisfying \eqref{assum:delta_less_eps}, $u_{\eps,\delta}$ is (locally) asymptotically stable in $\{u\,\vert u(\frac{\cdot}{k_c+\delta})\in Y\}$, where $Y$ is defined by \eqref{eq:space_of_perturbations} and $k_c$ is given by \eqref{assum:spectrum_intersects_zero_ii}--\eqref{assum:omega_non_zero}.
\end{thm}

It is evident that the last result also applies to suitable shifts:

\begin{cor}\label{cor:stability_of_shifts}
	Under conditions of Theorem~\ref{thm:lp_stability}, for any $h\in[0,2\pi)$, $u_{\eps,\delta}(\cdot-h)$ is asymptotically stable in $\{u\,\vert u(\frac{\cdot}{k_c+\delta})\in Y_h\}$, where
	\[
		Y_h=\{ l \cos(x-h) :l\in\R\} \oplus {\textnormal{ran}} \Ac \subset L^2_{2\pi}(\R).
	\]
\end{cor}

We start with a lemma on compactness of the convolution operator, which also explains the reason for the condition~\eqref{eq:ccstab} in Theorem~\ref{thm:lp_stability}.

\begin{lem}
\label{lem:l1_convol_compact_in_l2per}
	Let $a\in L^1(\R\to\R_+)$ and 
	\begin{equation}\label{eq:kernel_is_shifts_cont}
		\lim_{\tilde{x}\to0} \int[\R] |a(\tilde{x}+x)-a(x)| ~\txtd x =0.
	\end{equation}
	Then the operator $Ah = a\ast h$ is compact in $L^2_k(\R)$, for any $k>0$. 
\end{lem}
\begin{cor}
	The operator $A$ is compact in $Y$, since $A$ is invariant on $Y$ and $Y$ is a subspace of $L^2_{2\pi}(\R)$.
\end{cor}
\begin{proof}[Proof of Lemma \ref{lem:l1_convol_compact_in_l2per}]
	First note, that for any $h\in L^2_k(\R)$, we have $a\ast h \in L^2_k(\R)$ and
	\begin{align*}
		(a\ast h)(x) &= \int[\R] a(x-y)h(y)~\txtd y = \sum_{j\in\Z} \int[jk][jk+k] a(x-y) h(y)~\txtd y\\ 
		&= \int[0][k] \sum_{j\in\Z} a(x-y+jk)h(y)~\txtd y =: \int[0][k] b(x-y) h(y)~\txtd y,
	\end{align*}
	where $\|b\|_{L^1([0,k])} = \|a\|_{L^1(\R)}$ and $\|b\ast h\|_{L^2[0,k]}=\|a\ast h\|_{L^2_k(\R)}$. Let $\{h_j{\in}L^2_k(\R)\}_{j\in\N}$ be a uniformly bounded sequence of functions in $L^2_k(\R)$, namely $\|h_j\|_{L^2_k(\R)}\leq C$. We are going to check the conditions of the Fr\'{e}chet-Kolmogorov-Riesz theorem (see e.g.~\cite{Bre2011,Kru2018}) to show that $\{a\ast h_j\}_{j\in\N}$ is precompact, which is going to finish the proof. Indeed, by Young's convolution inequality and~\eqref{eq:kernel_is_shifts_cont}, we see that $\{b\ast h_j\}_j$ is uniformly bounded
	\[
		\|b\ast h_j\|_{L^2[0,k]} =  \|a\ast h_j\|_{L^2_k(\R)} \leq C \|a\|_{L^1(\R)} <\infty,
	\]
	and equicontinuous
	\[
		\|T_{\tilde{x}}(b\ast h_j) - b\ast h_j\|_{L^2[0,k]} = \|(T_{\tilde{x}}a\ast h_j) - a\ast h_j\|_{L^2_k(\R)} \leq C \|T_{\tilde{x}}a-a\|_{L^1(\R)} \to 0,
	\]
	as $\tilde{x}\to 0$, where $T_{\tilde{x}}h(y):= h(x+\tilde{x})$. 
\end{proof}

\begin{proof}[Proof of Theorem~\ref{thm:lp_stability}]
As before, we let $k=k_c+\delta$, where $k_c$ is given by assumptions \eqref{assum:spectrum_intersects_zero_ii}--\eqref{assum:omega_non_zero}. For a fixed $\eps$, which is defined as in Theorem \ref{thm:steady_state_bifurcation}, $F(v,\eps,k)$ is defined by \eqref{eq:basic_shifted_scaled}. We already know that $v_{\eps,\delta}(x) = u_{\eps,\delta}(\frac{x}{k}) - \theta$ satisfies $F(v_{\eps,\delta},\eps,k)=0$. Let us consider the linearization of $F$ at $v_{\eps,\delta}$, 
\begin{equation*}
	Lh := \frac{\partial F}{\partial v}(v_{\eps,\delta},\eps, k) h = \Aek h - \kame h (a_k^-\ast v_{\eps,\delta}) - \kame v_{\eps,\delta} (a_k^-\ast h), \quad h\in L^2_{2\pi}(\R),
\end{equation*}
where $\Aek$ is given by \eqref{eq:basic_shifted_scaled}.
First, we find the essential spectrum of $L$ in $L^2_{2\pi}(\R)$.
Since, by Lemma~\ref{lem:l1_convol_compact_in_l2per}, $h\to a^\pm\ast h$ are compact operators in $L^2_{2\pi}(\R)$, then by Weyl's theorem (see e.g.  \cite[Example~XIII.4.3]{RS1978iv}),
\begin{equation*}
	\sigma_{{\textnormal{ess}}}(L) = \{\la\in\C: \la+\kape\in \sigma_{{\textnormal{ess}}}(B),\ Bh:= -\kame (a_k^-\ast v_{\eps,\delta}) h\}.
\end{equation*}
An alternative characterization of the spectrum of $B$  (see e.g.~\cite[Problem~VII.17b]{RS1978i}) is given by the essential range of $-\kame (a_k^-\ast v_{\eps,\delta})$, which is, by Lemma~\ref{lem:lone_linft}, 
\[
	\sigma(B) = \sigma_{{\textnormal{ess}}}(B) = \{ -\kame (a_k^-\ast v_{\eps,\delta})(x) : x\in \R \}.
\]
Since, for all $\delta$ satisfying \eqref{assum:delta_less_eps}, $\|v_{\eps,\delta}\|_\infty\to0$ as $\delta\to0$ and $\eps\to0$, we find the existence of $\eps_0$ such that for all $0<\eps\leq \eps_0$,
\[
	\sigma_{{\textnormal{ess}}}(L) = \{-\kape -\kame(a_k^-\ast v_{\eps,\delta})(x): x\in\R \} \subset \{z\in\C:\text{Re}z<0\}.
\]
In summary, the essential spectrum cannot produce any instability as it is contained in the left-hal of the complex plane. Let us now study the discrete spectrum of $L$. We are looking for solutions to the eigenvalue problem $Lh = \la h$.
We apply the Lyapunov-Schmidt reduction method as we did in Section~\ref{sec:existence}. 
Let us consider the space decomposition \eqref{eq:space_of_perturbations} and the corresponding projection $P$ on $Y \cap \ker (\Ac)$, where $\Ac := \mathcal{A}_{0,k_c}$.
Then the eigenvalue problem 
\[
H(h,\eps,\delta,\la) := L(\eps,\delta)h-\la h = 0, \quad \text{for $h\in Y$}, 
\]
is equivalent to  
\begin{align}
	P H(y+\psi,\eps,\delta,\la) &= 0, \quad y = Ph,\ \psi = (1-P)h, \label{eq:reduced_periodic} \\
 (1-P) H(y+\psi,\eps,\delta,\la) &=0. \label{eq:implicit_periodic}
\end{align}
Since, for $y_0(x):=\cos(x)$ (that is the eigenvector to $L(0,0)$ and $\la=0$), the maps
\begin{align*}
	\txtD_\psi (1-P)H(y_0,0,0,0) &= (1-P)\Ac :{\textnormal{ran}}(\Ac)\to {\textnormal{ran}}(\Ac),\\
	\txtD_\la (1-P)H(y_0,0,0,0) &= - (1-P)\1: {\textnormal{ran}}(\Ac)\to {\textnormal{ran}}(\Ac),
\end{align*}
both are linear diffeomorphisms, we may apply the Implicit Function Theorem. In particular, there exist $\eps_0$, $\delta_0$, $r$, $\psi=\psi(y,\eps,\delta)$, $\la=\la(y,\eps,\delta)$, $\psi(0,0,0)=0$, $\la(0,0,0)=0$, such that for all $|\eps|<\eps_0$, $|\delta|<\delta_0$, $\|y\|\leq r$ we get 
\[
	(1-P)H(y_0+y+\psi(y,\eps,\delta),\eps,\delta,\la(y,\eps,\delta))=0.
\]
Moreover we evidently also obtain $\psi\in C^2$, $\la\in C^2$, and $\la(y,0,0)=0$ for $\|y\|\leq r$.
Similar to Section~\ref{sec:existence} the map $\psi$ satisfies an expansion for $\|y\|\leq r$, $|\eps|<\eps_0$, $|\delta|<\delta_0$, given by
\begin{equation}\label{eq:another_expansion_for_psi}
	\psi(y,\eps,\delta)= G_0(\eps,\delta) + G_1(y,\eps,\delta)+R(y,\eps,\delta),
\end{equation}
where $G_j(\cdot,\eps,\delta):(\ker (\Ac))^j\to {\textnormal{ran}}(\Ac)$, $j=0,1$, are $j$-forms with respect to $y$ and $R(y,\eps,\delta) = o(\|y\|^1)$ as $y\to0$. Since $\la(\cdot,0,0)\equiv0$ for $\|y\|\leq r$ one checks that
\[
	\la(y,\eps,\delta)= \la_0(\eps,\delta) + o((|\eps|+|\delta|)\|y\|), \quad \|y\|+|\delta|+|\eps|\to0,
\]
where $\la_0(\eps,\delta) = \cO (|\eps|+|\delta|)$, as $\delta\to0$, $\eps\to0$.
Collecting $0$-forms in \eqref{eq:implicit_periodic}, we get for $G_0=G_0(\eps,\delta)$, 
\begin{equation*}
	\Aek G_0 -\kame (1-P)G_0 (a^-_k\ast v_{\eps,\delta}) - \kame (1-P) v_{\eps,\delta} (a^-_k\ast G_0) = \la_0(\eps,\delta) G_0.
\end{equation*} 
As in Section~\ref{sec:existence}, the Implicit Function Theorem implies $G_0\equiv0$. Collecting $1$-forms we get 
\begin{align}
\label{eq:brack2}
	\Aek G_1(y) =&\ \la_0 G_1(y) + \kame(1-P)\left[ (a^-_k\ast v_{\eps,\delta})(y+G_1(y)) +v_{\eps,\delta} a^-_k\ast (y+G_1(y)) \right].
\end{align}
By \eqref{eq:periodic_solution}, $v_{\eps,\delta}= 2\sqrt{\frac{\Omega(\eps,\delta)}{\omega}} \cos(x) + \cO (|\eps-\delta^2|)$, where $\Omega(\eps,\delta)$ is given by \eqref{eq:alpha_series_decomposition}. Therefore, we have
\[
	(a^-_k\ast v_{\eps,\delta}) G_1(y,\eps,\delta) + v_{\eps,\delta} (a^-_k\ast G_1(y,\eps,\delta)) = o(\|G_1(y,\eps,\delta)\|), \quad \delta\to0,\ \eps\to0, 
\]
which deals with the lst two terms inside the brackets in~\eqref{eq:brack2}. For the first term we obviously have
\[
	\la_0(\eps,\delta)G_1(y,\eps,\delta) = o(\|G_1(y,\eps,\delta)\|), \quad \delta\to0,\ \eps\to0. 
\]
By \eqref{eq:perturbed_A3} and \eqref{eq:Aek_inverse}, there exists $c>0$ such that for all $f\in {\textnormal{ran}}(\Ac)$,
\[
	\frac{1}{c} \|\Ac^{-1} f\| \leq \|f\| \leq c\|\Ac^{-1}f\|.
\]
As a result we obtain
\begin{align}
	G_1(&y,\eps,\delta) \sim \kam \Ac^{-1} (1-P) \left[ (a^-_{k_c}\ast v_{\eps,\delta})y +v_{\eps,\delta} (a^-_{k_c}\ast y) \right] \nonumber \\
	 &\sim 2\kam\sqrt{\frac{\Omega(\eps,\delta)}{\omega}}\Ac^{-1}(1-P)\left[ \widehat{a^-}(k_c) \cos(x) y + \cos(x)(a^-_{k_c}\ast y) \right] \nonumber\\
	&\sim 2l\kam\widehat{a^-}(k_c)\sqrt{\frac{\Omega(\eps,\delta)}{\omega}} \Big(\frac{\cos(2x)}{\alpha(0,2k_c)}+\frac{1}{\alpha(0,0)}\Big),\quad |\delta|+|\eps|\to0,  \label{eq:G_1_asymptotic}
\end{align}
where $y(x) = l\cos(x)$. It remains to be checked what happens for the term $R=R(y,\eps,\delta)$ in~\eqref{eq:another_expansion_for_psi}. We claim that it is a remainder term, which satisfies $R=o(G_1)$ as $|\delta|+|\eps|\to0$. Collecting in \eqref{eq:implicit_periodic} terms of order $o(\|y\|)$, we have
\begin{align*}
	\Aek &R(y,\eps,\delta) - (\la(y,\eps,\delta)-\la_0(\eps,\delta)) G_1(y,\eps,\delta)\\ 
	&= \kame (1-P) \left[ R(y,\eps,\delta) (a^-_k\ast v_{\eps,\delta}) + v_{\eps,\delta} (a^-_k\ast R(y,\eps,\delta))\right] + \la(\eps,\delta) R(y,\eps,\delta)\\	 
	&= o(R(y,\eps,\delta)),\quad |\delta|+|\eps|\to0.
\end{align*} 
By solving the last equation to leading-order we easily find
\[
	R(y,\eps,\delta) \sim \Ac^{-1} (\la(y,\eps,\delta)-\la_0(\eps,\delta))G_1(y,\eps,\delta) = o(G_1(y,\eps,\delta)), \quad |\delta|+|\eps|\to0.
\]
As a result, for all $\|y\|\leq r$, we indeed have as claimed
\[
	\psi(y,\eps,\delta) \sim G_1(y,\eps,\delta) + o(\|G_1(y,\eps,\delta)\|), \quad |\delta|+|\eps|\to0.
\]
we can now use \eqref{eq:reduced_periodic}, which yields 
\begin{align*}
	0&= PH(y+\psi,\eps,\delta,\la)\\ 
	&= \alpha(\eps,k)y-\la(y,\eps,\delta)y -\kame P(a^-\ast v_{\eps,\delta})(y+\psi) - \kame P v_{\eps,\delta} a^-\ast (y+\psi), 
\end{align*}
where $\alpha(\eps,k)$ is defined by \eqref{eq:def_alpha}. Hence, as $|\delta|+|\eps|\to0$, we can calculate that
\[
	\alpha(\eps,k)y \sim \la y - \kame P \left[ (a^-_k\ast v_{\eps,\delta})(y+G_1(y,\delta,\eps)) +v_{\eps,\delta} a^-_k\ast (y+G_1(y,\eps,\delta)) \right].
\]
Then, by \eqref{eq:alpha_series_decomposition}, \eqref{assum:omega_non_zero}, \eqref{eq:G_1_asymptotic}, and since $P(\cos^2(x))=0$, 
\begin{align*}
	0\sim&\ \Omega(\eps,\delta) + \la(y,\eps,\delta) 
		+ (\kame)^2 \widehat{a^-}(k_c)  \Big(\frac{\widehat{a^-}(2k_c)+\widehat{a^-}(k_c)}{\alpha(0,2k_c)}+\frac{2+2\widehat{a^-}(k_c)}{\alpha(0,0)}\Big) \frac{\Omega(\eps,\delta)}{2\omega}\\ 
		\sim&\ \frac{3}{2}\Omega(\eps,\delta) + \la(y,\eps,\delta), \quad |\delta|+|\eps|\to0.
\end{align*}
Finally, by \eqref{assum:delta_less_eps} we obtain 
\[
	\la \sim -\frac{3}{2}\Omega(\eps,\delta) <0, \quad |\eps|+|\delta|\to0,
\]
which implies asymptotic stability of $v_{\eps,\delta}$, and, as a result, of $u_{\eps,\delta}$. Note that it could be possible that new eigenvalues of $L$ appear for $\eps>0$. Since $\|v_{\eps,\delta}\|\to 0$ as $\eps\to0$ and $\delta$ satisfying \eqref{assum:delta_less_eps}, then, by e.g. \cite[Theorem~I.2.2]{DK1974} such eigenvalues belong to a neighbourhood of $\sigma(\Ac)$ for small $\eps$, namely
\[
	\sigma(L) \subset \{ x+r : x\in\sigma(\Ac),\ |r|\leq R(\eps,\delta)\}, \quad R(\eps,\delta) = \cO (\sqrt{\Omega(\eps,\delta)}).
\]
By the Implicit Function Theorem applied above we can redefine $\eps_0>0$ such that there is no new eigenvalue around $0$ and thus in the positive half-space for all $\eps<\eps_0$.
\end{proof}


\section{Examples} 
\label{sec:examples}

We still have to show that there exist kernels satisfying all our assumptions so that we can get bifurcations. We are going to provide two examples. Both examples are motivated by the goal to find simple, yet non-trivial kernels, where can check our conditions.  

\begin{exmp}
\label{exmp:gaussian2}
We start with Gaussians, respectively linear combinations of Gaussians, and consider
	\[
		a^+(x) = \frac{1}{\sqrt{2\pi l}} \e^{-\frac{x^2}{2l}}; \qquad a^-(x) = \frac{1}{2\sqrt{2\pi q}}\big(\e^{-\frac{(x-h)^2}{2q}} + \e^{-\frac{(x+h)^2}{2q}}\big).
	\]
In this case, the Fourier transforms of $a^\pm$ have the following form for $p\in\R$,
	\[
		\widehat{a^+}(p) = \e^{-\frac{l p^2}{2}}; \qquad \widehat{a^-}(p) = \cos(hp)\e^{-\frac{q p^2}{2}}.
	\]
	We put $l=q=2$, $\kap=1$, $m\in(0,1)$, $\gamma:=\kap{-}m$. Then it is straightforward to verify that \eqref{assum:assumptions_basic} holds. Next, one can just calculate 
\[
	\alpha(0,p) = (1-\gamma \cos(hp)) \e^{-p^2} -1, \quad p\in\R.
\] 
Hence, $\alpha(0,p)=0$ holds if and only if $1-\gamma\cos(hp)=\e^{p^2}$. For sufficiently small $h>0$, $(1-\gamma\cos(hp))<\e^{p^2}$, $p\in\R$. Monotonically increasing $h$ one will find $h_c>0$ such that $1-\gamma\cos(h_c p)$ touches $\e^{p^2}$ at $p=k_c$. Moreover, such $k_c$ is unique in $\R_+$. Hence, \eqref{assum:spectrum_intersects_zero_ii} and \eqref{assum:alpha_doesnt_have_other_zeroes} hold.
Since, for $h=h_c$, $\alpha(0,p)\leq0$, $p\in\R$, and $\alpha(0,k_c)=0$, we have that $k_c$ is a maximum of $\alpha$ we get 
\[
	0=\partial_p \alpha(0,k_c) = \e^{-k_c^2}((-2k_c)(1-\gamma \cos(h k_c) ) + \gamma h \sin(h k_c)).
\]
We can also calculate the second derivative directly to see that
\begin{align*}
	\e^{k_c^2} \cdot \frac{\partial^2 \alpha}{\partial p^2} (0,k_c) &= -2(1-\gamma\cos(hk_c)) - 2k_c \gamma h \sin (hk_c)+ \gamma h^2 \cos(hk_c) \\
		&= -(2+4k_c^2)(1-\gamma \cos(hk_c)) +\gamma h^2 \cos(hk_c) <0,
\end{align*}
where we use the equality $1-\gamma \cos(hk_c) = \e^{k_c^2}>1$, which implies that $1-\gamma\cos(hk_c)>0$ and $\cos(hk_c)<0$.
As a result \eqref{assum:d2_alpha_dk2_less_zero} is satisfied. It appears to be complicated to check \eqref{assum:omega_non_zero} analytically. Therefore, we demonstrate~\eqref{assum:omega_non_zero} graphically for $m=0.5$ as shown in Figure~\ref{fig:omega}.

\begin{figure}[H]
  \includegraphics[width=\textwidth]{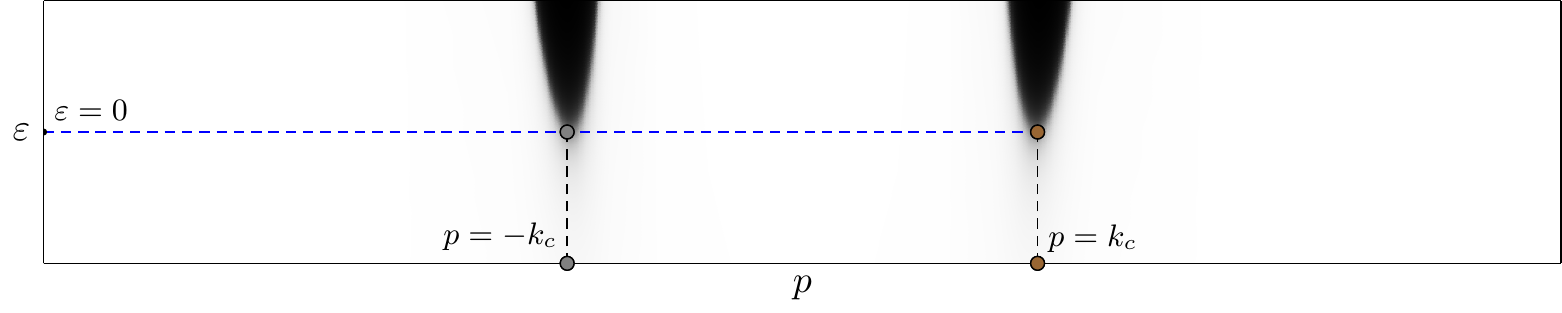}
  \caption{Computation of $\textnormal{sgn}(\alpha(\eps,p))$, where $-1$ is shown in white and $+1$ in black. This is shown only for illustration purposes and conditions on $\alpha$ can be checked analytically.}
\end{figure}

\begin{figure}[H]
  \includegraphics[width=\textwidth]{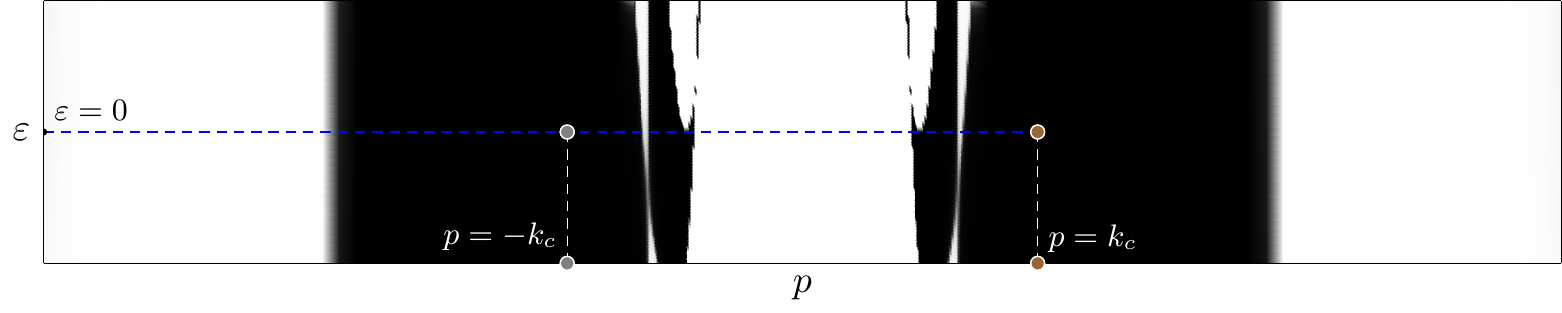}
  \caption{Computation of $\textnormal{sgn}(\omega(\eps,p))$. Again we show $-1$ in white and $+1$ in black.}
  \label{fig:omega}
\end{figure}

One can see on Figure~\ref{fig:omega} that $\omega=\omega(0,k_c)>0$. In fact, the condition is evidently not close to being violated in this case and the argument would be easy to make completely rigorous by just using interval arithmetic to validate the sign.
\end{exmp}

\begin{exmp}
\label{exmp:indicators}
The second example is in spirit similar to the first one, so we are a bit more brief. We consider uniform distributions:
	\[
		a^+(x) = \frac{1}{2l} \1_{[-l,l]}(x);  \qquad a^-(x) = \frac{1}{2q}\big( \1_{[-q-\tilde{h},-\tilde{h}]}(x) + \1_{[\tilde{h},q+\tilde{h}]}(x)\big).
	\]
In this case, the Fourier transform of $a^\pm$ has the following form, for $p\in\R$,
	\[
		\widehat{a^+}(p) = \frac{\sin(lp)}{lp}; \qquad \widehat{a^-}(p) = \frac{\sin(qp+\tilde{h}p) - \sin(\tilde{h}p)}{qp} = 2\frac{\cos(\tilde{h}p+\frac{qp}{2}) \sin(\frac{qp}{2})}{qp}.
	\]
We put $l=1$, $q=2$, $\kap=1$, $m\in(0,1)$, $\gamma:=\kap{-}m$. Then \eqref{assum:assumptions_basic} holds.
Next, for $h=\tilde{h}+1$, we find
\[ 
	\alpha(0,p) = (1-\gamma \cos(hp)) \frac{\sin(p)}{p}-1, \quad p\in\R.
\]
Since, for all $j\in\Z$, $\alpha(0,j\pi)\neq0$, then $\alpha(0,p)=0$ if and only if 
\[
	1-\gamma \cos(hp) = \frac{p}{\sin(p)}.
\]
For sufficiently small $h\geq1$, $(1-\gamma\cos(hp))< \frac{p}{\sin(p)}$, $p\in\R$, $p\neq j\pi$, $j\in\Z$.
Monotonically increasing $h$ one will find $h_c>0$ such that $1-\gamma\cos(h_c p)$ touches $\frac{p}{\sin(p)}$ at $p=k_c$.
Moreover, such $k_c$ is unique in $\R_+$. Hence, \eqref{assum:spectrum_intersects_zero_ii} and \eqref{assum:alpha_doesnt_have_other_zeroes} hold. Since, for $h=h_c$, $\alpha(0,p)\leq0$, $p\in\R$ and $\alpha(0,k_c)=0$, then $k_c$ is a maximum of $\alpha$ and 
\[
	0=\partial_p \alpha(0,k_c) = (1-\gamma \cos(h k_c))\frac{k_c\cos(k_c) - \sin(k_c)}{k_c^2} + \gamma h \sin(hk_c) \frac{\sin(k_c)}{k_c}.
\]
Hence, we calculate
\begin{align*}
	\frac{\partial^2 \alpha}{\partial p^2} (0,k_c) &= -(1-\gamma \cos(h k_c))\frac{\sin(k_c)}{k_c} +2\gamma h \sin(hk_c) \frac{k_c\cos(k_c)-\sin(k_c)}{k_c^2}\\ 
	&\quad + \gamma h^2 cos(hk_c) \frac{\sin(k_c)}{k_c} \\
	&= -1 - 2 \big( \frac{k_c\cos(k_c)-\sin(k_c)}{k_c^2}\big)^2 (1-\gamma \cos(hk_c))^2 \\ 
	&\quad + \gamma h^2\cos(hk_c)\frac{\sin(k_c)}{k_c}<0, 
\end{align*}
where we use the equality $1-\gamma \cos(hk_c) = \frac{k_c}{\sin(k_c)}>1$, which implies that $\cos(hk_c)<0$.
As a result, \eqref{assum:d2_alpha_dk2_less_zero} is satisfied. As in Example~\ref{exmp:gaussian2}, we check \eqref{assum:omega_non_zero} graphically for $m=0.5$.

\begin{figure}[H]
  \includegraphics[width=\textwidth]{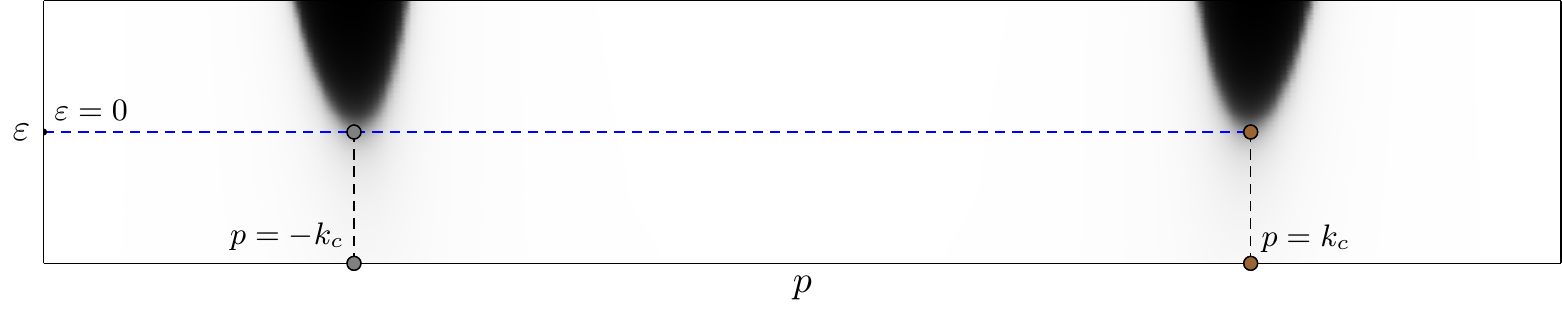}
  \caption{Computation of $\textnormal{sgn}(\alpha(\eps,p))$; same conventions as for plots above.}
\end{figure}

\begin{figure}[H]
  \includegraphics[width=\textwidth]{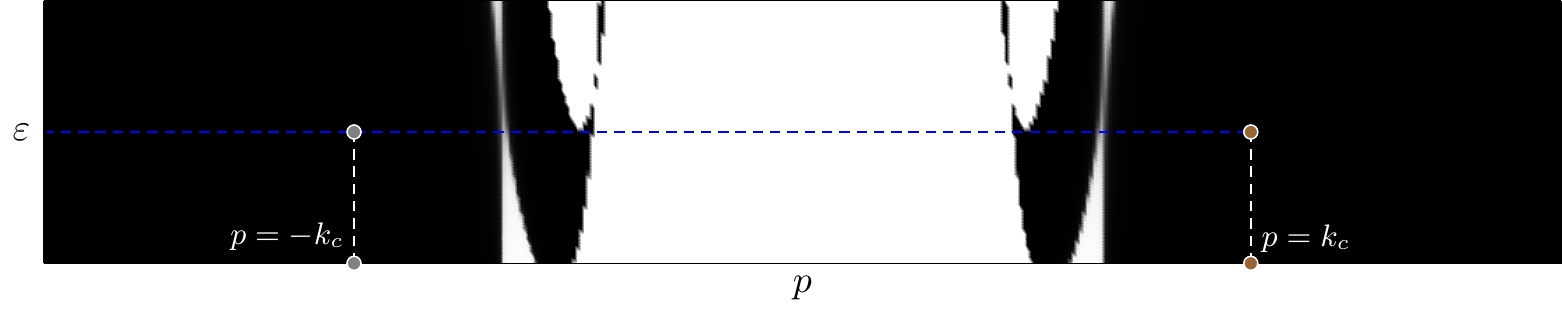}
  \caption{Computation of $\textnormal{sgn}(\omega(\eps,p))$; same conventions as for plots above.}
  \label{fig:omega2}
\end{figure}

One can again clearly see on \eqref{fig:omega2} that $\omega=\omega(0,k_c)>0$.  
\end{exmp}


\section{Relation to the Fisher-KPP equation with a non-local reaction} 
\label{sec:relation_to_FKPP_with_nonloc_reaction}

In this section we establish the connection between Theorem \ref{thm:steady_state_bifurcation} and 
and \cite[Theorem~1.1]{FH2015} for the nonlocal Fisher-KPP equation. For the convenience of the reader 
we are going to formulate \cite[Theorem~1.1]{FH2015} here again for reference. We consider the equation
\begin{equation}\label{eq:laplace_nonloc_reaction_normalized}
	\partial^2_{x} U(x) + \mu U(x)(1 - (a^-\ast U)(x)) = 0, \quad x\in\R,
\end{equation}
where $\mu>0$. We need to discuss the relevant hypotheses before stating the result.

\begin{hypothesis}\label{FH:H1}
The kernel $a^-$ satisfies:
\begin{equation*}
	a^-\geq0,\quad a^-(0)>0,\quad a^-(-x)\equiv a^-(x),\quad \int[\R] a^-(x)~\txtd x =1,\quad \int[\R] x^2 a^-(x)~\txtd x<\infty. 
\end{equation*}
\end{hypothesis}
Then one defines the usual dispersion relation
\begin{equation}\label{eq:FH:def_of_d}
	d(\mu,k):= -k^2-\mu \widehat{a^-}(k).
\end{equation}
\begin{hypothesis}\label{FH:H2}
	For $a^-(x)$ satisfying Hypothesis \ref{FH:H1}, further assume there exist unique $k_c>0$ and $\mu_c>0$ such that the following conditions are satisfied:
	\begin{enumerate}[label=(\roman*)]
		\item\label{FH:i}\begin{center}$d(\mu_c,k_c) = 0.$\end{center} 
		\item\label{FH:ii}\begin{center} $\partial_k d(\mu_c,k_c)=0.$\end{center}
		\item\label{FH:iii}\begin{center} $\partial^2_{k} d(\mu_c,k_c)<0.$\end{center}
		\item\label{FH:iv}\begin{center} $d(\mu_c,jk_c)\neq0,\quad \Z\ni j\neq\pm1.$\end{center}
	\end{enumerate}
\end{hypothesis}

Now we can state \cite[Theorem~1.1]{FH2015}:

\begin{thm}
\label{thm:FH:1_1}
Assume the hypotheses \ref{FH:H1}-\ref{FH:H2} above are satisfied and $\omega_1$ defined by \eqref{eq:FH:omega} is positive.
Let $\mu:= \mu_c +\tilde{\eps}$ and $k:=k_c+\delta$.
Then, there exists $\tilde{\eps}_0>0$, such that for all $\tilde{\eps}\in(0,\tilde{\eps}_0]$ and all 
\[
\delta^2 < \frac{-\widehat{a^-}(k_c)}{1+\frac{\mu_c}{2}\partial^2_{k}\widehat{a^-}(k_c)}\tilde{\eps} 
\]
there is a stationary $\frac{2\pi}{k}$-periodic solution of \eqref{eq:laplace_nonloc_reaction_normalized} with the leading expansion of the form 
\begin{equation*}
	U_{\tilde{\eps},\delta}(x) = 1 + 2\sqrt{\frac{-\widehat{a^-}(k_c)\tilde{\eps}-(1+\frac{\mu_c}{2}\partial^2_{k}\widehat{a^-}(k_c))\delta^2}{\omega_1}} cos\big((k_c+\delta)x\big) + o(|\tilde{\eps}|^{\frac12} + |\delta|),
\end{equation*}
where 
\begin{equation}\label{eq:FH:omega}
	\omega_1 := -\mu_c\widehat{a^-}(k_c) \left( \frac{\mu_c(\widehat{a^-}(k_c)+\widehat{a^-}(2k_c))}{4k_c^2+\mu_c\widehat{a^-}(2k_c)} +2(1+\widehat{a^-}(k_c))\right)>0.
\end{equation}
\end{thm}

To relate the nonlocal and the doubly-nonlocal results, we first note a useful preliminary formal computation
\begin{multline*}
	\frac{\kap}{\sigma^2}(a^+_\sigma\ast u-u) = \frac{\kap}{\sigma^2} \int_\R \frac{1}{\sigma}a^+(\frac{y}{\sigma})(u(x-y)-u(x))~\txtd y \\ 
	= \frac{\kap}{\sigma^2} \int_\R a^+(y)(u(x-\sigma y)-u(x))~\txtd y  \sim \gamma \kap\partial^2_{x} u(x) + o(\sigma),\quad \sigma \to 0.
\end{multline*}
where $\gamma := \frac{1}{2}\int_\R y^2 a^+(y) ~\txtd y$. Therefore, one can conjecture that the scaling limit of \eqref{eq:basic} is actually the following equation
\begin{equation}\label{eq:laplace_nonloc_reaction}
	\gamma \kap \partial^2_{x} u(x) + (\kap{-}m) u(x) -\kam u(x)(a^-\ast u)(x) = 0, \quad x\in\R.
\end{equation}
This motivates us to first rescale \eqref{eq:basic} suitably. We consider transformed parameters instead of $\kap_\eps$, $\kam_\eps$, $m$ and $a^+$ given by
\begin{gather*}
	\widetilde{\ka}^+_\eps(\sigma, \ka) = (1+\eps)\frac{\kap+\ka}{\sigma^2}, \quad \widetilde{\ka}^-_\eps(\sigma, \ka) = \big( 1+\eps\frac{ \kap + \ka}{\sigma^2 (\kap{-}m)}\big)\kam \\
	\widetilde{m}(\sigma, \ka) = m + \frac{\kap + \ka}{\sigma^2} - \kap, \quad a^+_\sigma(x) = \frac{1}{\sigma}a^+(\frac{x}{\sigma}),
\end{gather*}
where the dependence on $\sigma$, $\eps$ and $\ka$ is chosen so that 
\[
	\frac{\widetilde{\ka}^+_\eps - \widetilde{m}}{\widetilde{\ka}^-_\eps} = \frac{\kap-m}{\kam} = \theta \quad \text{and} \quad \widetilde{\ka}^+_0-\widetilde{m} = \kap - m. 
\]
Hence we arrive to the following equation
\begin{equation}\label{eq:basic_rescaled}
	\widetilde{\ka}^+_\eps (a^+_\sigma\ast u)(x) - \widetilde{m}u(x) - \widetilde{\ka}^-_\eps u(x) (a^-\ast u)(x) = 0, \quad x\in\R.
\end{equation}
We want to compare our results applied to~\eqref{eq:basic_rescaled} to the (singular) limit $\sigma\rightarrow 0$, where the results of Theorem~\ref{thm:FH:1_1} hold. We denote for $\sigma>0$ (cf.~\eqref{eq:def_alpha})  
\begin{equation*}
	\at(\eps,p,\sigma,\ka) = \widetilde{\ka}^+_\eps \widehat{a^+_\sigma}(p) - (\widetilde{\ka}^+_\eps - \widetilde{m}) \widehat{a^-}(p) - \widetilde{\ka}^+_\eps,
\end{equation*}
and assume $\int[\R] x^2a^+(x)~\txtd x <\infty$ (cf.~\eqref{assum:assumptions_basic}). We extend $\tilde{\alpha}$, for $\eps=0$ and $\sigma\leq0$, as follows
\begin{align*}
	\tilde{\alpha}(0,p,\sigma,\ka) &:= \lim_{\sigma\to 0_+} \tilde{\alpha}(0,p,\sigma,\ka) = - \gamma \kap p^2 - (\kap{-}m) \widehat{a^-}(p),  \quad \sigma\leq0. 
\end{align*}
Note that the extension is continuous in $\sigma$, i.e., $\sigma{\mapsto}\tilde{\alpha}(0,p,\sigma,\ka)$ is in $C(\R)$. Now one can repeat the formulations of the assumptions \eqref{assum:assumptions_basic}--\eqref{assum:omega_non_zero} in terms of the transformed parameters.
We simply label these assumptions as \eqref{assum:assumptions_basic}$_\sigma$--\eqref{assum:omega_non_zero}$_\sigma$. In particular, by \eqref{assum:spectrum_intersects_zero_ii}$_\sigma$--\eqref{assum:delta_less_eps}$_\sigma$ we define $k_c(\sigma)$, and by \eqref{assum:omega_non_zero}$_\sigma$ we define $\omega(\sigma)$. The next result states that we indeed obtain a natural limiting result if we let the linear part of the doubly-nonlocal problem converge to the classical diffusion case.  

\begin{thm}\label{thm:connection_to_FH2015}
Let $\gamma=1$, $\omega_1$ given by \eqref{eq:FH:omega} be positive, \eqref{assum:assumptions_basic} and Hypothesis~\ref{FH:H1} hold true, additionally $\kap,m>0$ be such that Hypothesis~\ref{FH:H2} holds with $\mu_c=\frac{\kap{-}m}{\kap}$ and for some $k_c>0$.
Then there exists $\eps_0>0$ such that:
\begin{enumerate}
	\item[(T1)]\label{itm:connction_to_FH2015_statement1} For all $\eps\in(0,\eps_0)$ and $\delta^2 {<} \frac{-\widehat{a^-}(k_c)}{1+\frac{\mu_c}{2}\partial^2_{k}\widehat{a^-}(k_c)}\frac{m\eps}{\kap(1+\eps)}$ there exists a $\frac{2\pi}{k_c+\delta}$-periodic solution $u_{\eps,\delta}$ to \eqref{eq:laplace_nonloc_reaction} with the leading expansion of the form, 
  \begin{equation*}
		u_{\eps,\delta}(x) = \theta + 2\sqrt{\frac{\Omega_0(\eps,\delta)}{\omega_0}} cos\big((k_c+\delta)x\big) +o(|\eps|^{\frac12} + |\delta|),
  \end{equation*}
	where $\omega_0 \theta^2 = \omega_1$ and $\Omega_0(\eps,\delta) = -m \widehat{a^-}(k_c)\eps - \kap(1+\frac{\mu_c}{2})\partial^2_{k} \widehat{a^-}(k_c) \delta^2$.
\item[(T2)]\label{itm:connction_to_FH2015_statement2} For some $\sigma_0>0$ and for all $\sigma\in(0,\sigma_0)$ there exist $\ka(\sigma)$, $k_c(\sigma)$, such that \eqref{assum:assumptions_basic}$_\sigma$--\eqref{assum:d2_alpha_dk2_less_zero}$_\sigma$, \eqref{assum:omega_non_zero}$_\sigma$ hold true and for all $\eps\in(0,\eps_0)$ and all $\delta$ that satisfy \eqref{assum:delta_less_eps}$_\sigma$ there exists a $\frac{2\pi}{k_c(\sigma)+\delta}$-periodic solution $u_{\eps,\delta,\sigma}(x)$ to \eqref{eq:basic_rescaled} (with  $\ka=\ka(\sigma)$) with the leading expansion of the form
  \begin{equation*}
		u_{\eps,\delta,\sigma}(x) = \theta + 2\sqrt{\frac{ \Omega_\sigma(\eps,\delta)}{\omega(\sigma)}} cos((k_c(\sigma)+\delta)x) +o(|\eps|^{\frac12} + |\delta|),  
	\end{equation*}
as $\delta\to0,\ \eps\to0$, where $\omega(\sigma)$ is defined by \eqref{assum:omega_non_zero}$_\sigma$ and $\Omega_\sigma$ is given by (cf.~\eqref{eq:alpha_series_decomposition})
\[
		\Omega_\sigma(\eps,\delta) = \partial_\eps \tilde{\alpha} (0,k_c(\sigma),\sigma,\ka(\sigma))\eps + \frac{1}{2}\partial^2_{k} \tilde{\alpha}(0,k_c(\sigma),\sigma,\ka(\sigma))\delta^2.
	\]
\item[(T1)]\label{itm:connction_to_FH2015_statement3} We have $\ka(\sigma)\to 0$, $k_c(\sigma)\to  k_c$, $\omega(\sigma) \to \omega_0$ and $\Omega_\sigma(\eps,\delta)\to \Omega_0(\eps,\delta)$, as $\sigma\to0$.  As a result, for $\eps,\delta$ satisfying \ref{itm:connction_to_FH2015_statement1},
	\[
		\lim_{\sigma\to0} \|u_{\eps,\delta,\sigma} - u_{\eps,\delta}\| = o(|\eps|^{\frac12} + |\delta|),\quad \delta\to0,\ \eps\to0.
	\]
	\end{enumerate}
\end{thm}

\begin{proof}
	First note that if $U(x)$ solves \eqref{eq:laplace_nonloc_reaction_normalized} and $\gamma=1$, then $u(x)=\theta U(x)$ solves \eqref{eq:laplace_nonloc_reaction}, where $\mu = \frac{\kap{-}m}{\kap}$, $\theta = \frac{\kap-m}{\kam}$.
	Next, since $\omega_1>0$ and Hypotheses~\ref{FH:H1} and \ref{FH:H2} hold, then Theorem~\ref{thm:FH:1_1} implies the statement \ref{itm:connction_to_FH2015_statement1} with 
	\begin{equation}\label{eq:eps_tilde_eps}
		\tilde{\eps}=\frac{m\eps}{\kap(1+\eps)} \sim \frac{m}{\kap}\eps+\cO (\eps^2),\quad \eps\to0. 
	\end{equation}
 Next we apply the Implicit Function Theorem to the following equation at $(0,k_c,0,0)$,
 \begin{equation}\label{eq:k_and_sigma}
	\begin{cases}
	\widetilde{\alpha}(0, k, \sigma, \ka) = 0,\\
	\partial_k \widetilde{\alpha}(0, k, \sigma, \ka) =0.
\end{cases}
\end{equation}
By Hypothesis~\ref{FH:H2}, the following Jacobian matrix is non-degenerate at $(0,k_c,0,0)$,
\begin{equation}\label{eq:determinant}
	\frac{1}{2\kap} \lim_{\substack{\sigma \to 0_+\\ \ka\to 0\ }} \left| 
	 	\begin{array}{cc} \partial_\ka \at & \partial_k \at \\
											\partial_{\ka k} \at & \partial^2_{k} \at
		\end{array}
		\right| = k_c \partial_k d(0,k_c,\mu_c) - \frac{k_c^2}{2} \partial^2_{k} d(0,k_c,\mu_c) > 0,
\end{equation}
where the function $d$ is defined by \eqref{eq:FH:def_of_d}. Indeed, we calculate
\begin{align*}
	\partial_\ka \at(0,k,\sigma,\ka) &= \frac{\widehat{a^+}(k\sigma)-1}{\sigma^2} \to - \gamma k^2,\quad \sigma\to0_+,\\
	\partial_k \at(0,k,\sigma,\ka) &= (\kap{+}\ka) \int_\R \frac{-\i y\sigma}{\sigma^2} \e^{-\i k\sigma y} a^+(y) ~\txtd y - (\kap{-}m) \int_\R (-\i y) \e^{-\i ky}a^-(y) ~\txtd y\\
	&\to - 2\kap\gamma k + (\kap{-}m) \int_\R  \i y \e^{-\i ky}a^-(y)~\txtd y, \quad \sigma\to0_+,\ \ka \to 0,
\end{align*}
\begin{align}
	\partial_{\ka k} \at(0,k,\sigma,\ka) &= \int_\R \frac{-\i y\sigma}{\sigma^2} \e^{-\i k\sigma y} a^+(y) ~\txtd y \to -2\gamma k,\quad \sigma\to0_+, \nonumber \\
	\partial^2_{k} \at(0,k,\sigma,\ka) &= (\kap{+}\ka) \int_\R \frac{(-\i y\sigma)^2}{\sigma^2} \e^{-\i k\sigma y} a^+(y) ~\txtd y \nonumber \\ 
	&\quad - (\kap{-}m) \int_\R (-\i y)^2 \e^{-\i ky}a^-(y) ~\txtd y, \nonumber \\ 
	&\to - 2\kap \gamma + (\kap{-}m) \int_\R  y^2 \e^{-\i ky}a^-(y)~\txtd y, \label{eq:d_at_d_kk}
\end{align}
as $\sigma\to0_+$, $\ka\to0$. 
Since $\gamma=1$, then at evaluating at $(0,k_c,0,0)$ these results yield
\begin{align*}
	\frac{1}{2\kap}[\partial_\ka\at\, \partial^2_{k k}\at - \partial_k \at\, \partial^2_{\kap k} \at] &= - \gamma^2 k_c^2 +\gamma\mu_c\int_\R (\frac{k_c^2 y^2}{2} + \i k_c y) \e^{-\i k_c y}a^-(y)~\txtd y \\ 
	&= k_c \partial_k d(0,k_c,\mu_c) - \frac{k_c^2}{2} \partial^2_{k} d(0,k_c,\mu_c) > 0.
\end{align*}
Thus \eqref{eq:determinant} is proven and by the Implicit Function Theorem there exist 
\[
\sigma_0>0\quad\text{ and }\quad k_c(\sigma),\ka(\sigma)\in C((-\sigma_0,\sigma_0)\to\R),
\]
which solve \eqref{eq:k_and_sigma} for all $\sigma\in(-\sigma_0,\sigma_0)$. Moreover, $k_c(0)=k_c$ and $\ka(0)=0$. Next, by \eqref{eq:k_and_sigma} and \eqref{eq:d_at_d_kk} we see that as $\sigma\to 0$ one has
\begin{align}
	\partial_\eps \tilde{\alpha}(0,k_c(\sigma),\sigma,\ka(\sigma)) &= \tilde{\alpha}(0,k_c(\sigma),\sigma,\ka(\sigma)) - m\widehat{a^-}(k_c(\sigma))  \to -m \widehat{a^-}(k_c); \label{eq:d_at_d_eps}\\ 
	\partial^2_{k} \tilde{\alpha}(0,k_c(\sigma),\sigma,\ka(\sigma)) &\to -2\kap - (\kap-m) \partial^2_{k} \widehat{a^-}(k_c) \nonumber \\
	&\qquad = -2 \kap(1+\frac{\mu_c}{2})\partial^2_{k}\widehat{a^-}(k_c). \label{eq:d2_at_d_kk}
\end{align}
As a result the limit $\sigma\to0_+$ indeed gives us
\begin{align*}
	\Omega_\sigma(\eps,\delta)\to -m \widehat{a^-}(k_c)\eps - \kap(1+\frac{\mu_c}{2})\partial^2_{k} \widehat{a^-}(k_c) \delta^2 = \Omega_0(\eps,\delta),
\end{align*}
which shows that the leading-order coefficient converges. Nest, we note that
\begin{equation}\label{eq:alpha_and_d}
	\tilde{\alpha} (0,k,0,0) = \kap d(\mu_c,k), \quad \mu_c = \frac{\kap-m}{\kap}. 
\end{equation}
Let us ensure that \eqref{assum:assumptions_basic}$_\sigma$--\eqref{assum:omega_non_zero}$_\sigma$ hold for all $\sigma\in(0,\sigma_0)$: 
\begin{itemize}
	\item \eqref{assum:assumptions_basic} and \eqref{assum:assumptions_basic}$_\sigma$ are equivalent; 
	\item \eqref{assum:spectrum_intersects_zero_ii}$_\sigma$ and \eqref{eq:dalpha_dk_zero}$_\sigma$ follow from \eqref{eq:k_and_sigma}; 
 	\item by \eqref{eq:alpha_and_d}, Hypothesis~\ref{FH:H2}~\ref{FH:iv} and $d(\mu_c,-\infty) = -\infty$ (possibly redefining $\sigma_0$) \eqref{assum:alpha_doesnt_have_other_zeroes}$_\sigma$ holds;
	\item similarly, \eqref{eq:alpha_and_d} and Hypothesis~\ref{FH:H2}~\ref{FH:iii} imply \eqref{assum:d2_alpha_dk2_less_zero}$_\sigma$;
\end{itemize}
Therefore Theorem~\ref{thm:steady_state_bifurcation} yields statement \ref{itm:connction_to_FH2015_statement2}.
Since $k_c(\sigma)\to k_c$ as $\sigma\to0$, then $\omega(\sigma)\to\omega_0$ as $\sigma\to0$.
To finish the proof of the statement \ref{itm:connction_to_FH2015_statement3} it is left to notice that by \eqref{eq:eps_tilde_eps}, \eqref{eq:d_at_d_eps} and \eqref{eq:d2_at_d_kk}  (c.f. \eqref{assum:delta_less_eps}$_\sigma$)
\[
	\tilde{\eps} \frac{ \partial_\mu d(\mu_c,k_c)}{-\frac{1}{2}\partial^2_{k} d(\mu_c,k_c)} = \tilde{\eps} \frac{-\widehat{a^-}(k_c)}{1+\frac{\mu_c}{2}\partial^2_{k}\widehat{a^-}(k_c)} \leq \frac{\eps m}{\kap} \frac{-\widehat{a^-}(k_c)}{1+\frac{\mu_c}{2}\partial^2_{k}\widehat{a^-}(k_c)}=	\eps	\frac{2\partial_\eps \at(0,k_c,0,0)}{-\partial^2_{k} \at(0,k_c,0,0)}.
\]
The proof is fulfilled.
\end{proof}

\begin{rem}\label{rem:on_scaling}
	It is worth to point out that typically a diffusive scaling is considered for $\ka=0$. We introduce $\ka\in\R$ to get an additional `degree of freedom' that allows us to choose $\ka=\ka(\sigma)$ such that the spectrum of the linearization of \eqref{eq:basic_rescaled} at $u\equiv\theta$ touches the imaginary axis \textit{for all} small $\sigma>0$ (c.f. \eqref{eq:k_and_sigma}).
\end{rem}


\section{On nonexistence of stationary solutions} 
\label{sec:nonexistence}

One might now ask, whether all the assumptions are really crucial to obtain a non-trivial bifurcating solution. Here we provide several results to indicate that one can easily find other parameter regimes, where no bifurcations occur. For the convenience of the reader, we formulate here a special case of \cite[Proposition 5.12]{FKT2015}.

\begin{prop}\label{prop:nonexistence_of_stationary_sol}
	Let the following assumptions hold
	\begin{gather*}
		\kap>m, \quad a^\pm(-x) \equiv a^\pm(x), \quad a^\pm\in L^\infty(\R),\quad \int[\R]|x|a^+(x)~\txtd x<\infty,\\ 
		\kap a^+(x) \geq (\kap{-}m)a^-(x),\ x\in\R,\quad \kap a^+(0) > (\kap{-}m)a^-(0)>0, 
	\end{gather*}
	Then there exist only two non-negative bounded solutions to \eqref{eq:basic}, namely $u\equiv0$ and $u\equiv\theta$.
\end{prop}

In fact, one can even describe that nothing can happen ``between'' the two homogeneous solutions, even for other parameter ranges as the following results shows:

\begin{prop}
	Let $a^+\in L^1(\R\to\R_+)$ be such that
	\[
		\int[\R] a^+(y)~\txtd y = 1, \qquad \int[\R] |y| a^+(y)~\txtd y<\infty.
	\]
	Then, for any $l\in(0,\theta)$, there does not exist a non-zero solution $u\in L^\infty(\R)$ to \eqref{eq:basic}, such that $0 \leq u(x) \leq l$, $x\in\R$.
\end{prop}

\begin{proof}
We argue by contradiction and suppose there exists $u\in L^\infty(\R)$ satisfying \eqref{eq:basic} and $0\leq u(x) \leq l$ for $x\in\R$. Then, we must have
	\[
		0 \geq \kap (a^+\ast u)(x) -\kam l u(x) - m u(x), \quad x\in\R.
	\]
This implies
	\begin{equation}\label{eq:subsol_to_lin_eq}
		- \frac{\kam}{\kap}(\theta - l) u(x) \geq (a^+\ast u)(x) - u(x), \quad x\in\R.
	\end{equation}
We distinguish two cases. Suppose first that $u\in L^1(\R)$, then we get
	\[
		0\geq - \frac{\kam}{\kap} (\theta-l) \int[\R]u(x)~\txtd x \geq \int[\R] \int[\R] a^+(x-y)u(y)~\txtd y~\txtd x - \int[\R] u(x)~\txtd x = 0, 
	\]
which implies $u\equiv 0$ as $a^+$ has mass one. For the second case let $u\not\in L^1(\R)$. For any $r>0$ we compute 
\begin{align*}
		\int[-r][r] \big( &(a^+\ast u)(x) - u(x) \big) ~\txtd x = \int[\R]a^+(y) \int[-r][r] (u(x-y) - u(x)) ~\txtd x~\txtd y\\
			&\geq \int[|y|\leq r] a^+(y) \Big( \int[r][r-y] u(x)~\txtd x - \int[-r-y][-r] u(x)~\txtd x \Big) - 2r \|u\|_\infty \int[|y|>r] a^+(y)~\txtd y\\
			&\geq -2 \|u\|_\infty \int[|y|\leq r] a^+(y)|y| ~\txtd y -2 r \|u\|_\infty \int[|y|>r] \frac{|y|}{r}a^+(y)~\txtd y = -2\|u\|_\infty \int[\R] |y| a^+(y) ~\txtd y.
	\end{align*}
As a result, by \eqref{eq:subsol_to_lin_eq},
	\[
		-\frac{\kam}{\kap}(\theta-l)\cdot\infty \geq -2\|u\|_\infty \int[\X] |y|a^+(y)~\txtd y > -\infty,
	\]
where the left-hand side is infinite because $u\not\in L^1(\R)$. Therefore, we have obtained again a contradiction.
\end{proof}

To describe the stationary solution set also near $u\equiv \theta$, we need an auxillary result. The following theorem follows from \cite[$5.1.6$ and Remark $5.1.1$]{OR1970}:

\begin{thm} 
\label{thm:homeomorphism}
Let E be a Banach space, $A, A^{-1}$ be linear continuous operators in $E$, $G:E \to E$, such that
\begin{align*}
		0< c < \|A^{-1}\|^{-1};\qquad \|Gx-Gy\|_E \leq c \|x-y\|_E,\ x,y\in B_\delta(x_0),
\end{align*}
where $B_\delta(x_0) = \{ x\in E : \|x-x_0\|_E \leq \delta\}$. Then $A-G$ is a homeomorphism between $B_\delta(x_0)$ and $(A-G)(B_\delta(x_0))$.
\end{thm}

\begin{prop}\label{prop:local_nonexist_at_theta}
Let $p>0$, $\alpha$ be defined by \eqref{eq:def_alpha} and 
\begin{equation}\label{eq:alpha_is_negative}
		\gamma_p := -\sup_{j\in\Z} \alpha(0,\tfrac{2\pi j}{p}) >0.
\end{equation}
	Then for any $\delta<\frac{\gamma_p}{2\kam I_p(a^-)}$, there does not exist a solution to \eqref{eq:basic} in
\[
	\{u\in L^2_p(\R) : \|u-\theta\|_{L^2_p(\R)} \leq \delta\}\backslash\{\theta\}.
\]
\end{prop}

\begin{proof}
Consider $w=u-\theta$. If $u$ satisfies \eqref{eq:basic}, then $w$ satisfies \eqref{eq:basic_shifted}. We apply Theorem \ref{thm:homeomorphism} to the following operators
	\[
		Af = \kap a^+\ast f - (\kap{-}m)a^-\ast f- \kap\ast f, \quad Gf = \kam fa^-\ast f.
	\]
	By \eqref{eq:def_alpha}, for any $f\in L_p^2(\R)$ we have a Fourier decomposition 
	\[
		f = \sum_{j\in\Z} f_j \e^{\i \frac{2\pi j}{p} x},\ f_j= (f,\e^{\i \frac{2\pi j}{p} \cdot })_{L_p^2(\R)}; \quad Af = \sum_{j\in\Z} f_j \alpha(0,\tfrac{2\pi j}{p}) \e^{\i \frac{2\pi j}{p}x}.
	\]
	Hence, by \eqref{eq:alpha_is_negative}, we can also compute the inverse
	\[
		A^{-1}f = \sum_{j\in\Z} \frac{f_j }{\alpha(0,\tfrac{2\pi j}{p})} \e^{\i \frac{2\pi j}{p} x}; \quad \|A^{-1}\|  = \sup_{j\in\Z} \frac{1}{| \alpha(0,\tfrac{2\pi j}{p}) |} = \frac{1}{\gamma_p}. 
	\]
	By \eqref{eq:l2p_linf_est}, for any $f,g$, $\|f\|_{L_p^2(\R)} \leq \delta$, $\|g\|_{L_p^2(\R)} \leq \delta$, 
	\begin{align*}
		\tfrac{1}{\kam} \|Gf-Gg\|_{L_p^2(\R)} &\leq \|a^-{\ast }f\|_\infty \|f-g\|_{L^2_p(\R)} + \|g\|_{L_p^2(\R)} \|a^-{\ast }(f{-}g)\|_\infty \\ 
		&\leq 2 \delta  I_p(a^-) \|f-g\|_{L_p^2(\R)}.
	\end{align*}
	Thus, for any $\delta < \frac{\gamma_p}{2\kam I_p(a^-)}$, $A-G: B_\delta(0) \to (A-G)(B_\delta(0))$ is a homeomorphism. The proof is fulfilled.
\end{proof}

\begin{prop}\label{prop:radius_of_uniqueness_around_theta}
	Assume that for $J_\theta := \kap a^+ - \theta \kam a^-$, $\|J_\theta\|_1<\kap$.
	Then, for any $\delta < \frac{\kap - \|J_\theta\|_1}{2\kam}$, there does not exist a solution to \eqref{eq:basic} in
  \[
  	\{u\in L^\infty(\R) : \|u-\theta\|_\infty \leq \delta\}\backslash\{\theta\}.
  \]
\end{prop}
\begin{proof}
	Consider $w=u-\theta$. If $u$ satisfies \eqref{eq:basic}, then $w$ satisfies \eqref{eq:basic_shifted}.
	We apply Theorem \ref{thm:homeomorphism} to the following operators
	\[
		Af = J_\theta\ast f - \kap f, \qquad Gf = \kam fa^-\ast f.
	\]
	Since $\sigma(A)\subset B_{\|J\|_1}(-\kap)$, then $\sigma(A^{-1})\subset \{\la : \frac{1}{\la}\in B_{\|J\|_1}(-\kap)\}$. Hence
	\[
		\|A^{-1}\| \leq \sup\{ |\la|: \frac{1}{\la} \in B_{\|J_\theta\|_1}(-\kap)\} = \frac{1}{\kap-\|J_\theta\|_1}.
	\]
	For any $f,g \in B_\delta(0)$,
	\[
		\|Gf-Gg\|_\infty \leq \kam (\|f\|_\infty+\|g\|_\infty)\|f-g\|_\infty \leq 2\kam \delta\|f-g\|_\infty.
	\]
	Thus for any $\delta < \frac{\kap - \|J_\theta\|_1}{2\kam}$, $A-G: B_\delta \to (A-G)(B_\delta)$ is a homeomorphism. The proof is fulfilled.
\end{proof}

The results in this section show that there are also many cases, where bifurcations are impossible.\medskip

\textbf{Acknowledgements:} CK would like to thank the VolkswagenStiftung for support via a Lichtenberg Professorship.
PT wishes to express his gratitude to the ``Bielefeld Young Researchers'' Fund for the support through the Funding Line Postdocs: ``Career Bridge Doctorate -- Postdoc''.

\end{document}